\documentclass[11pt]{amsart}   	
\usepackage[margin=1in]{geometry}
\usepackage{subfigure}					
\usepackage{amsmath}
\usepackage{amssymb, amsfonts}
\usepackage[nobysame]{amsrefs}
\usepackage{amsthm}
\usepackage{hyperref}
\usepackage{mathtools}

\usepackage[shortlabels]{enumitem}
\usepackage{comment}
\usepackage{Andrew}
\newcommand{\precdot}{\prec\mathrel{\mkern-5mu}\mathrel{\cdot}}

\usepackage[noabbrev,capitalise,nameinlink]{cleveref}

\definecolor{lavender}{rgb}{0.4,0,1.0}

\hypersetup{colorlinks=true, citecolor=lavender, linkcolor=lavender,urlcolor=lavender}

\makeatletter
\newtheorem*{rep@theorem}{\rep@title}
\newcommand{\newreptheorem}[2]{%
\newenvironment{rep#1}[1]{%
 \def\rep@title{#2 \ref{##1}}%
 \begin{rep@theorem}}%
 {\end{rep@theorem}}}
\makeatother

\newtheorem{theorem}{Theorem}[section]
\newreptheorem{theorem}{Theorem}

\newtheorem{proposition}[theorem]{Proposition}
\newtheorem{corollary}[theorem]{Corollary}
\newtheorem{conjecture}[theorem]{Conjecture}
\newtheorem{question}[theorem]{Question}

\newtheorem{lemma}[theorem]{Lemma}

\theoremstyle{definition}
\newtheorem{definition}[theorem]{Definition}
\newtheorem{remark}[theorem]{Remark}
\newtheorem{example}[theorem]{Example}

\newcommand{\OEIS}[1]{{\rm \href{http://oeis.org/#1}{\texttt{#1}}}}

\newcommand{\BB}{\mathsf{B}}
\renewcommand{\AA}{\mathsf{A}}
\renewcommand{\Z}{\mathbb{Z}}
\renewcommand {\tand}{\text{ and }}
\newcommand{\ATam}{\mathsf{ATam}}
\newcommand{\CSymATam}{\mathsf{CSymATam}}
\newcommand{\CTam}{\mathsf{CTam}}

\newcommand{\TT}{\mathsf{T}}

\newcommand{\Weak}{\mathrm{Weak}}
\newcommand{\OO}{\mathcal{O}}

\DeclareMathOperator{\Bicl}{Bicl}

\newcommand{\Inv}{\mathrm{Inv}}
\newcommand{\Tam}{\mathrm{Tam}}

\newcommand{\join}{\vee}
\newcommand{\bigjoin}{\bigvee}
\newcommand{\meet}{\wedge}
\newcommand{\bigmeet}{\bigwedge}
\newcommand{\opp}{\mathrm{op}}

\newcommand{\IndexingSet}{\mathcal{I}}
\newcommand{\DirectedSet}{\mathcal{A}}
\newcommand{\DirectLimit}{\varinjlim}
\newcommand{\InverseLimit}{\varprojlim}
\DeclareMathOperator{\PointwiseUnionClosure}{PUC}

\newcommand{\dfn}[1]{\textcolor{blue}{\emph{#1}}}

\usepackage{thmtools}

\title{Lattices from Pointed Building Sets: Generalized Ornamentation Lattices}

\author[]{Andrew Sack}
\address[]{Department of Mathematics, University of Michigan, Ann Arbor, MI 48109, USA}
\email{{\href{mailto:asack@umich.edu}{asack@umich.edu}}}

\begin{document}


\begin{abstract}

We introduce a novel combinatorial structure called \emph{pointed building sets}, which can be viewed as families of lattices equipped with compatibility relations. To each pointed building set $\BB$, we associate a complete lattice $\OO(\BB)$, called the \emph{ornamentation lattice} of $\BB$.

Special cases of this construction have already proven useful in understanding the structure of three families of posets: operahedron lattices, the affine Tamari lattice, and hypergraphic posets of subhypergraphs of the path hypergraph of an increasing tree.

The goal of this paper is to establish the theory of these generalized ornamentations. We examine several natural classes of pointed building sets which recover classical lattices such as the Tamari lattice, the lattice of topologies ordered by coarsening, and the lattice of naturally labeled partial orders. Furthermore, several theoretical directions are explored, including inverse limits and group actions. Notably, this leads to a straightforward construction of inverse limits of Tamari lattices, yielding infinite analogs of the Tamari lattice.

\end{abstract}

\maketitle

\section{Introduction}

There has been an explosion of study of combinatorial lattice theory in recent decades. In many cases, proving that a particular class of posets form lattices requires a great deal of effort~\cite{barkley2023affine, BarnardMcConville, defant2024operahedron, vonbell2024framing}. In this paper, we introduce the concept of a \emph{pointed building set}, a simple combinatorial object that is reminiscent of the \emph{building sets} of De Concini and Procesi~\cite{de1995wonderful}. Building sets were introduced to understand the ``wonderful'' compactifications of hyperplane arrangement complements, and have since become a central tool in the study of matroids~\cite{adiprasito2018hodge, Feichtner2005} and generalized permutohedra~\cite{postnikov2008faces, Postnikov_GP}.  

Pointed building sets are collections of pointed sets satisfying natural closure properties.  To each pointed building set $\BB$ we associate a poset of functions called \emph{ornamentations}, which will turn out to be a lattice. Recently, special cases of ornamentation lattices have served as a tool to dramatically simplify proofs that several distinct posets are lattices or to otherwise understand their structure~\cite{abram2025ornamentation, barkley2025affine, defant2024operahedron}. Special cases of ornamentation lattices arising from natural pointed building sets include the Tamari lattice, the affine Tamari lattice of Barkley and Defant~\cite{barkley2025affine}, the lattice of topologies ordered by coarsening, and the lattice of naturally labeled partial orders.
We will now present the main definitions of this paper in order to expand upon this background.

\begin{definition}
\label{def:pointed_building_set}
    Let $\IndexingSet$ be an indexing set. A \dfn{pointed subset} $(S, i)$ of $\IndexingSet$ is an ordered pair such that $S \subseteq \IndexingSet$ and $i \in S$. We say that $(S, i)$ is \dfn{pointed} at $i$. As an abuse of notation, it will generally be useful to treat a pointed set as its underlying set. For example, we will say that $j \in (S, i)$ if $j \in S$ and that $(S, i) \subseteq (T, j)$ if $S \subseteq T$. If $\mathcal C$ is a collection of pointed subsets of $\IndexingSet$, we define $$\mathcal C|_i := \{(S, j) \in \mathcal C \mid j = i\}.$$
    
    A \dfn{pointed building set} $\BB$ is a collection of pointed subsets of $\IndexingSet$ such that the following three axioms hold:
    \begin{enumerate}
        \item $\BB$ \emph{contains all singletons}, i.e., for all $i \in \IndexingSet$ we have $(\{i\}, i) \in \BB$;
        \item $\BB$ is \emph{transitively closed}, i.e., for all $(S, i), (T, j) \in \BB$ if $j \in S$ then $(S \cup T, i) \in \BB$;
        \item $\BB$ is \emph{closed under arbitrary pointwise unions}, i.e., for all $i \in \IndexingSet$, the collection $\BB|_i$ is closed under arbitrary (non-empty) unions.
    \end{enumerate}

\end{definition}    
Note that axiom (2) implies axiom (3) if $\IndexingSet$ is finite. For $i \in \IndexingSet$, note that axioms (1) and (3) guarantee that $\BB|_{i}$ is a complete lattice ordered by inclusion. See~\cref{fig:pointed_building_set_example} for an example of a pointed building set.

\begin{figure}
    \centering
    \includegraphics[scale = 1.25]{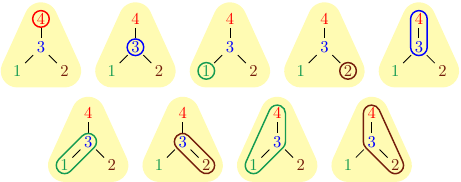}
    \caption{An example of a pointed building set. In this example, all pointed sets are pointed at their minimal element. We color the outline of a set with the same color as its pointed element.}
    \label{fig:pointed_building_set_example}
\end{figure}

\begin{example}
\label{ex:digraphical_pointed_building_set}
    Let $D = (V, E)$ be a directed graph. For $v \in V$, a subset $S \subseteq V$ is called \dfn{$v$-connected} if for all $u \in S$ there exists a directed path in $S$ from $v$ to $u$. Let $$\BB(D) := \{(S, v) \mid \text{$S$ is $v$-connected in $D$}\}.$$    
    We call $\BB(D)$ the \dfn{digraphical pointed building set} of $D$. \cref{fig:pointed_building_set_example} is a digraphical pointed building set when all edges are oriented upwards.
\end{example}

\begin{definition}
\label{def:ornamentation}
    Given an indexing set $\IndexingSet$ and a pointed building set $\BB$ on $\IndexingSet$, an \dfn{ornamentation} of $\BB$ is a function $\rho: \IndexingSet \to \BB$ such that

    \begin{enumerate}
        \item \emph{$\rho(i)$ is pointed at $i$}, i.e., for all $i \in \IndexingSet$, $\rho(i) \in \BB|_{i}$;
        \item $\rho$ is \emph{transitively closed}, i.e., for all $i, j \in \IndexingSet$ if $j \in \rho(i)$ then $\rho(j) \subseteq \rho(i)$.
    \end{enumerate}
Let $\OO(\BB)$ be the set of ornamentations of $\BB$.
\end{definition}

$\OO(\BB)$ has a natural partial order, where $\rho_1 \preceq \rho_2$ if for all $i \in \IndexingSet,$ we have $\rho_1(i) \subseteq \rho_2(i)$.   The following is our main result and justifies calling $(\OO(\BB), \preceq)$ the \dfn{ornamentation lattice} of $\BB$. 

\begin{figure}
    \centering
    \includegraphics[scale = 1.25]{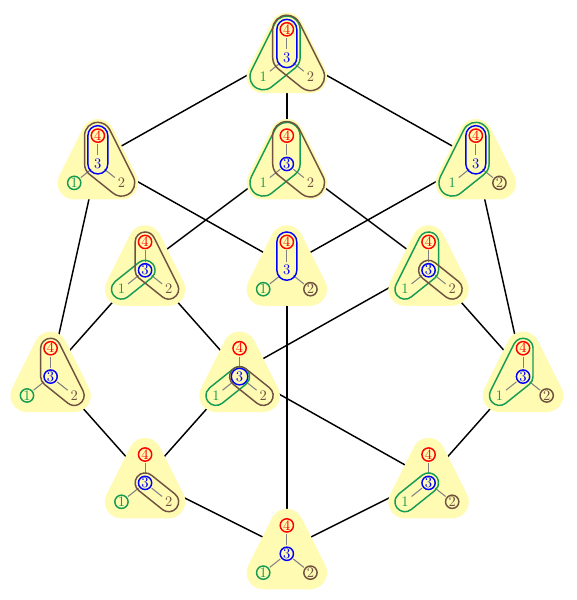}
    \caption{The ornamentation lattice of the pointed building set in~\cref{fig:pointed_building_set_example}.}
    \label{fig:ornamentations_example}
\end{figure}

\begin{theorem}
\label{thm:complete_lattice}
    $(\OO(\BB), \preceq)$ is a complete lattice.
\end{theorem}
We delay the proof of \cref{thm:complete_lattice} until \cref{sec:Lattice_Properties}. See~\cref{fig:ornamentations_example} for the ornamentation lattice of~\cref{fig:pointed_building_set_example}.  Note that for an ornamentation $\rho$, it is not necessarily true that $\rho(i) \tand \rho(j)$ are nested or disjoint.

In~\cite{defant2024operahedron}, ornamentations were used to prove that the operahedron poset of a rooted plane tree $\mathrm{Op}(\TT)$~\cite{laplante2022diagonal}*{Definition 2.8} is a lattice, confirming a conjecture of Laplante-Anfossi. In particular, an embedding $\mathrm{Op}(\TT) \hookrightarrow \mathcal L(\TT) \times \OO(\BB(\TT))$ provided a critical structural lemma that enabled the proof, where $\mathcal L(\TT)$ is the lattice of linear extensions of $\TT$.

Ornamentation lattices of rooted trees were further studied in~\cite{ajran2025pop}, where they were shown to have several nice combinatorial properties such as \emph{semidistributivity}. The ornamentations defined in this paper are an attempt to generalize ornamentations of rooted trees defined in~\cite{defant2024operahedron} while still retaining the simple combinatorics that make them useful. 

As a demonstration that these generalizations are useful, an early draft of this paper was communicated to the authors of~\cite{barkley2025affine} who used our generalized ornamentations to understand the structure of the \emph{affine Tamari lattice}, which turns out to be isomorphic to $\OO(\BB(C_n))$, the digraphical ornamentation lattice of an $n$-cycle. This isomorphism helped construct a classification of the possible lengths of maximal green sequences for the completed path algebra of the oriented $n$-cycle. This will be discussed further in~\cref{sec:Examples} and~\cref{sec:group_actions}. 

A further application of generalized ornamentations was found in~\cite{abram2025ornamentation}. In particular, let $\mathbb H$ be a hypergraph with vertex set $[n] := \{1, \dots, n\}$. For a subset $X \subseteq [n]$, define the simplex $$\Delta_X := \operatorname{conv} \{e_i \mid i \in X\}.$$ The hypergraphic polytope $\Delta_{\mathbb H}$ is the Minkowski sum $\Delta_{\mathbb H} := \sum_{E \in \mathbb H} \Delta_E$.
The Hasse diagram of the \emph{hypergraphic poset} $P_{\mathbb H}$ is a particular orientation of the 1-skeleton of $\Delta_{\mathbb H}.$ In~\cite[Proposition 6.27]{abram2025ornamentation}, a special case called \emph{intreeval hypergraphic posets} was studied. For a directed (but not necessarily rooted) tree $D$ with vertex labels increasing along paths, let $\mathbb H(D)$ be the hypergraph consisting of all directed paths of $D$, and let $\mathbb H'$ be a subhypergraph of $\mathbb H(D)$. A surjective map $\OO(\BB(D)) \to P_{\mathbb H'}$ was used to completely classify when the hypergraphic poset of $\mathbb H'$ is a lattice.

The main goal of this paper is to introduce and explore the theory of ornamentation lattices. The structure of this paper is as follows. In~\cref{sec:Preliminaries}, we review necessary notation, terminology, and background for the remainder of the paper. In~\cref{sec:Examples}, we give numerous examples of both pointed building sets and ornamentation lattices. In~\cref{sec:Lattice_Properties}, we prove various poset and lattice-theoretic properties about ornamentation lattices.
    
In~\cref{sec:Digraphical_Ornamentation_Lattices}, we study the special case of digraphical ornamentations with a focus on directed trees.  This culminates in our second main result.

\begin{reptheorem}{thm:tree_duality}[Directed tree duality]
    Let $D$ be a directed tree. Then $\OO(\BB(D^\opp)) \simeq \OO(\BB(D))^\opp$. 
\end{reptheorem}

In~\cref{sec:direct_limits}, we do some light category theory and show that, in some sense, direct limits of pointed building sets induce inverse limits of ornamentation lattices. As a consequence, we obtain simple models for inverse limits of Tamari lattices. In~\cref{sec:infinite_Tamari}, we investigate a connection between 312-avoiding total orders and ornamentation lattices. In~\cref{sec:group_actions}, we study ornamentations that are invariant under group actions on their base set. This allows us to provide a simple description of the \emph{cyclic Tamari lattice}~\cite{barkley2025affine}, which answers a question of Barkley and Defant. This description enables a more direct proof of the maximal length of a chain in the cyclic Tamari lattice~\cite{barkley2022combinatorial}*{Theorem 8.4}.
In~\cref{sec:geometric_interpretations}, we provide some geometric interpretations of ornamentation lattices. Finally, in~\cref{sec:questions}, we end with some questions and further directions to explore.

\section{Preliminaries}
\label{sec:Preliminaries}

We assume familiarity with the theory of posets (partially ordered sets), a standard reference for which is~\cite{Stanley}*{Chapter 3}. We further define $[n] := \{1, \dots, n\}$.

Let $P = (X, \preceq)$ be a poset and $D = (V,E)$ be a directed graph. 
The \dfn{dual} of $P$ (denoted $P^\opp = (X, \preceq_\opp)$) is a poset on $X$ with $x \preceq_\opp y$ if and only if $y \preceq x$. The \dfn{dual} of $D$ (denoted $D^\opp$) is the directed graph on the same vertex set as $D$ but with all edge orientations reversed. We say $D$ is a \dfn{directed tree} if the underlying undirected graph of $D$ is a tree. A \dfn{directed path} in $D$ is a sequence of vertices (possibly containing only 1 element) $v_1, \dots, v_n$ such that there is an edge $(v_i, v_{i+1})$ for each $1 \le i \le n-1$.
If $D$ is acyclic, we define a poset $(V, \preceq_D)$ by $u \preceq_D v$ if and only if there exists a directed path from $u$ to $v$ in $D$. 

For $x,y\in P$ with $x\preceq y$, the \dfn{(closed) interval} from $x$ to $y$ is the set $$[x,y]:=\{z\in P:x\preceq z\preceq y\},$$ which we view as a subposet of $P$. Further define the \dfn{open interval} $$(x,y):=\{z\in P:x\prec z\prec y\}$$ and the \dfn{half-open} intervals 
    $$\begin{aligned} 
    [x,y)&:=\{z\in P:x\preceq z\prec y\}\\
    (x,y]&:=\{z\in P:x\prec z\preceq y\}.
    \end{aligned}
    $$ 
For convenience, we define $[x,x) := \{x\}$ in this paper.

If $x\prec y$ and $[x,y]=\{x,y\}$, then we say $y$ \dfn{covers} $x$ and write $x\precdot y$. For $x\in P$, let \[\Delta_P(x)=\{z\in P:z\preceq x\}\quad\text{and}\quad\nabla_P(x)=\{z\in P:x\preceq z\}.\] At times it will be convenient to use the notation $[x, \infty)$ in place of $\nabla_P(x)$.

A poset $L$ is called a lattice if any pair of elements $x, y \in L$ has a least upper bound called the \dfn{join} and denoted $x \join y$ and a greatest lower bound called the \dfn{meet} and denoted $x \meet y$. A lattice is called \dfn{complete} if any arbitrary subset $\Omega \subseteq L$ has a least upper bound $\bigjoin \Omega$ and greatest lower bound $\bigmeet \Omega$. Note that all finite lattices are complete.

A lattice $L$ is called \dfn{meet-semidistributive} if for all elements $x,y,z\in L$ satisfying $x\meet y=x\meet z$, we have $x\meet y=x\meet(y\join z)$. We say $L$ is \dfn{join-semidistributive} if for all $x,y,z\in L$ satisfying $x\join y=x\join z$, we have $x\join y=x\join(y\meet z)$. We call $L$ \dfn{semidistributive} if $L$ is both meet-semidistributive and join-semidistributive. 
An element $x \in L$ is called \dfn{join-irreducible} if for all finite non-empty subsets $\Omega \subseteq L$ such that $\bigjoin \Omega = x$ we have $x \in \Omega$.

Let $P$ be a poset with unique minimal element $\hat 0$. An element $x \in P$ is called an \dfn{atom} if it covers $\hat 0$.
    A complete lattice is called \dfn{atomic} (or sometimes \dfn{atomistic}) if every element is a join of a (possibly infinite) set of atoms.

Let $\IndexingSet$ be a finite set. A \dfn{building set}~\cite{de1995wonderful, Postnikov_GP} $X$ on $\IndexingSet$  is a collection of non-empty subsets of $\IndexingSet$ satisfying \begin{enumerate}
        \item for all $i \in \IndexingSet$, $\{i\} \in X$; 
        \item for all $I, J \in X$, if $I \cap J \neq \emptyset$, then $I \cup J \in X$.
    \end{enumerate}

\section{Examples}
\label{sec:Examples}

Every finite lattice is an ornamentation lattice for some pointed building set as the following example shows. 
\begin{example}
Let $X$ be a set and let $Y \subseteq 2^X$ be any collection of subsets of $X$ that is closed under unions. We take $\IndexingSet = X \sqcup \{\hat 0\}$ where $\sqcup$ denotes disjoint union. Define $$\BB = \{(\{i\}, i) \mid i \in \IndexingSet\} \cup \{(Z \sqcup \{\hat 0\}, \hat 0) \mid Z \in Y\}.$$

It is known that every finite lattice $L = (X, \preceq)$ is isomorphic (as a poset) to a union-closed family of sets ordered under inclusion. In particular, consider the union-closed family $Y = \{S_x \mid x \in X\}$ where $$S_x = \{y \in L \mid x \not\preceq y\}.$$ Then $\OO(\BB)$ is isomorphic to $L$.

\end{example}

As all lattices are ornamentation lattices, one should instead view this as a construction for combining the different lattices $\BB|_i$ to make a new lattice. However, in practice there are many natural classes of pointed building sets and therefore natural classes of ornamentation lattices.

\begin{example}
Let $G = (V,E)$ be a graph. Let $$\BB(G) = \{(C,i) \subseteq V \mid C \text{ is connected in $G$ and } i \in C\}.$$
    We call $\BB(G)$ the \dfn{graphical pointed building set} of $G$.
    
\end{example}

\begin{example}
    
Let $X$ be an (unpointed) building set on $\IndexingSet$. Then take $$\BB(X) := \{(S, i) \mid S \in X \tand i \in S\}.$$ When $\IndexingSet$ is the set of vertices of a finite graph and $X$ is the \emph{graphical} building set of $\IndexingSet$, i.e., the set of all connected subgraphs, then $\BB(X)$ is the graphical pointed building set of $\IndexingSet$.
\end{example}

\begin{example}
Let $X$ be an (unpointed) building set on $[n]$. Then take $$\BB(X) := \{(S, i) \mid S \in X \tand i = \min S\}.$$
\end{example}

\begin{example}
\label{ex:TamariLatticeAndInfiniteTamariLattice}
    For $a, b \in \Z$, let $[a,b] := \{i \in \Z \mid a \le i\le b\}$ and let $[a, \infty) := \{i \in \Z \mid a \le i\}$.    

    \begin{enumerate}
        \item If $\IndexingSet = [n]$ and  $$\BB = \{([a, b], a) \mid a, b \in \Z \st 1 \le a \le b \le n \},$$ then $\Tam_n := \OO(\BB)$ is the \dfn{Tamari lattice} as proven by Huang and Tamari~\cite{HUANG19727}. 

        \item If $\IndexingSet = \Z^+$ and $$\BB = \{([a, b], a) \mid a, b \in \Z^+ \st a \le b \} \cup \{([a, \infty), a) \mid a \in \Z^+ \},$$ then we call $\OO(\BB)$ the \dfn{infinite Tamari lattice} and denote it $\Tam_\infty$. 
        \item If $\IndexingSet = \Z$ and $$\BB = \{([a, b], a) \mid a, b \in \Z \st a \le b \} \cup \{([a, \infty), a) \mid a \in \Z \},$$ then we call $\OO(\BB)$ the \dfn{bi-infinite Tamari lattice} and denote it $\Tam_{\pm \infty}$. 
    \end{enumerate}

    Note that each of these Tamari lattices is a digraphical ornamentation lattice of a (possibly infinite) directed path. We will study $\Tam_\infty$ and $\Tam_{\pm \infty}$ in much greater generality in~\cref{sec:infinite_Tamari}.
 
\end{example}    

\begin{example}
\label{ex:ContinuousTamariLattice}
        If $\IndexingSet$ is a closed interval in $\R$, $$\BB = \{([a, b), a) \mid a, b \in \IndexingSet \st a \le b\}$$ is a pointed building set. We call $\OO(\BB)$ the \dfn{continuous Tamari lattice.}
\end{example}

\begin{example}
    Let $K_n$ be the complete graph on $n$ vertices and let $\BB(K_n)$ be the graphical pointed building set on $K_n$. Then $\OO(\BB(K_n))$ is in bijection with the set of transitive digraphs on $[n]$ which is enumerated in~\cite{OEIS} by the sequence~\OEIS{A000798}. A transitive digraph is a directed graph such that
    \begin{itemize}
        \item whenever $(u, v)$ and $(v, w)$ are edges, so is $(u, w)$ and
        \item for all vertices $u$, $(u, u)$ is an edge.
    \end{itemize}
    To build a transitive digraph from an ornamentation $\rho$, take the edge set $$\{(i, j) \mid i \in [n] \tand j \in \rho(i)\}.$$
    Conversely, given a transitive digraph $D$, let $\rho_D(i)$ be the set of all $j \in [n]$ that are reachable from $i$ via a directed path (including $i$ itself).

    It is known that transitive digraphs are in bijection with topologies~\cite{evans1967computer} as follows. Let $D$ be a transitive directed graph on $[n]$. Then the collection $\{\rho_D(i) \mid i \in [n]\}$ forms a basis for the corresponding topology. This shows that $\OO(\BB(K_n))$ is isomorphic to the lattice of topologies ordered by coarsening. 
\end{example}

\begin{example}
Direct the edges of $K_n$ from $i$ to $j$ for $i < j$ and let $\BB(K_n)$ be the digraphical pointed building set of $K_n$. Then $\OO(\BB(K_n))$ is in bijection with the set of naturally labeled posets which is enumerated in~\cite{OEIS} by the sequence~\OEIS{A006455}. To obtain a naturally labeled poset from an ornamentation $\rho \in \OO(\BB(K_n))$, take relations $i \le j$ for all $j \in \rho(i)$. This recovers the lattice of natural partial orders of Avann~\cite{avann1972lattice}.
\end{example}

\begin{example}
\label{ex:aff_tam}
    Let $C_n$ be the oriented cycle graph whose vertex set is $[n]$ and whose edge set is $$\{(i,\; i+1) \mid 1 \le i \le n-1\} \cup \{(n, 1)\},$$ and let $\BB(C_n)$ be the digraphical pointed building set on $C_n$. Then $\OO(\BB(C_n))$ is isomorphic to the \dfn{affine Tamari lattice} $\ATam_n$ of Barkley and Defant~\cite{barkley2025affine}*{Proposition 8.3}. The Hasse diagram of this lattice is an orientation of the type-D associahedron.
    
    An early draft of this paper was communicated to Barkley and Defant, and this result was used to calculate that the maximal length of a chain in the affine Tamari lattice is $\binom{n + 1}{2}$.

\end{example}

\section{Lattice properties}
\label{sec:Lattice_Properties}

\subsection{Proof of Theorem \ref{thm:complete_lattice}}

\begin{remark}
\label{rmk:isomorphic_copy}
    For each $i \in \IndexingSet$, we will identify $\BB|_i$ with an isomorphic copy of $\BB|_i$ in $\OO(\BB)$ as follows. For $(S, i) \in \BB|_i$, define $$
    \rho_{(S,i)}(j) := \begin{cases}
        (S, i) & j = i \\ (\{j\}, j) & j \neq i.
    \end{cases}$$

    Furthermore, this copy is a principal order ideal in $\OO(\BB)$ with maximum $\rho_{\max(\BB|_i)}$.
\end{remark}

\begin{lemma}
    For all pointed building sets $\BB$, the poset $(\OO(\BB), \preceq)$ is bounded, that is, has a unique minimal and unique maximal element.
\end{lemma}
\begin{proof}
It immediately follows that $\rho: \IndexingSet \to \BB$ defined by $\rho(i) = (\{i\}, i)$ is an ornamentation and is minimal. 
    Define $\sigma: \IndexingSet \to \BB$ by $\sigma(i) = \max(\BB|_i).$
    We claim that $\sigma$ is an ornamentation. Indeed, suppose that $j \in \sigma(i)$. Then $(\sigma(i) \cup \sigma(j), i) \in \BB$ which is contained in $\sigma(i)$ because $\sigma(i)$ is maximal in $\BB|_i$. It is immediate that $\sigma$ is maximal by construction.
\end{proof}

\begin{proof}[Proof of \cref{thm:complete_lattice}]
As $(\OO(\BB), \preceq)$ is bounded, it suffices to show that it is a complete meet-semilattice. Let $\Omega \subseteq \OO(\BB)$. We claim that $\bigwedge \Omega$ is the pointwise meet of $\Omega$. More precisely, define $\lambda: \IndexingSet \to \BB$ by $\lambda(i) := \bigwedge_{\rho \in \Omega} \rho(i)$ where the meet of each $\rho(i)$ is taken in $\BB|_i$. 

Alternatively, for $i \in \IndexingSet$ define $ \Omega|_{i} := \bigcap_{\rho \in \Omega} \rho(i)$. Then
 $$\lambda(i) = \left(\bigcup\limits_{\substack{S \in \BB|_{i} \\ S \subseteq \Omega|_{i}}} S,\ i \right).$$
That is, $\lambda(i)$ is the largest pointed set in $\BB$ contained in $\Omega|_{i}$ pointed at $i$.

We claim that $\lambda$ is an ornamentation. Let $j \in \lambda(i)$. Then for all $\rho \in \Omega$ we have $j \in \rho(i)$. We conclude that 
$\lambda(j) \subseteq \rho(j) \subseteq \rho(i),$
and so $\lambda(j) \subseteq \Omega|_i$.

As $(\lambda(i), i),\; (\lambda(j),j) \in \BB$ we have that $\left(\lambda(i) \cup \lambda(j),\; i\right) \in \BB$. Then by the maximality of $\lambda(i)$, we have that $\lambda(j) \subseteq \left(\lambda(i) \cup \lambda(j)\right) \subseteq \lambda(i)$. Hence $\lambda$ is an ornamentation.

For all $\rho \in \Omega$ and $i \in \IndexingSet$, $\lambda(i) \subseteq \rho(i)$, and so $\lambda \preceq \rho$. Furthermore, if $\sigma \in \OO(\BB)$ such that for all $\rho \in \Omega$ and $i \in \IndexingSet$ we have $\sigma(i) \subseteq \rho(i)$, then $\sigma(i) \subseteq  \Omega|_i$ and hence $\sigma(i) \subseteq \lambda(i)$.

Hence $\lambda = \bigmeet \Omega$.
\end{proof}

While the previous theorem suffices to show that $(\OO(\BB), \preceq)$ is a complete lattice, we can also explicitly construct the join.

\begin{definition}
Fix $\Omega \subseteq \OO(\BB)$ and let $\sigma: \IndexingSet \to \BB$ be defined by $$\sigma(i) = \left(\bigcup\limits_{\rho \in \Omega} \rho(i),\,\, i\right).$$  
\end{definition}

One may observe that $\sigma$ is \emph{not} necessarily an ornamentation, but is a well-defined function as each $\sigma(i)$ is a union of sets pointed at $i$.  See~\cref{fig:sigmaNonOrnamentation} for an example of $\sigma$. We will use $\sigma$ to construct $\bigjoin \Omega$.

\begin{figure}[h]
    \centering
    \includegraphics[scale = 1]{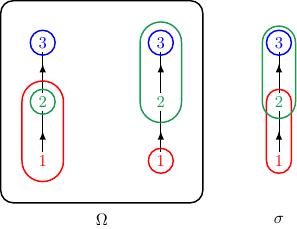}
    \caption{An example of $\sigma$ in $\Tam_3$. Observe that $\sigma$ is not necessarily an ornamentation.}
    \label{fig:sigmaNonOrnamentation}
\end{figure}

\begin{definition}
        Define a directed graph $D(\sigma)$  on $\IndexingSet$ with edge set $$\{\left(i,\; j\right) \mid j \in \sigma(i)\}.$$ For $i \in \IndexingSet$, let $$\pi(i) := \{j \in \IndexingSet \mid \text{there exists a directed path from $i$ to $j$ in }D(\sigma)\}.$$ 
\end{definition}

\begin{definition}

    Define $\nu: \IndexingSet \to \BB$ by $$\nu(i) := \left(\bigcup\limits_{j \in \pi(i)} \sigma(j),\,\, i\right).$$ 
    
\end{definition}

We will show that $\nu = \bigjoin \Omega$.

\begin{lemma}
\label{lem:path}
    If $i \in \IndexingSet$, then for all $j \in \nu(i)$, there is a path from $i$ to $j$ in $D(\sigma)$. 
\end{lemma}
\begin{proof}
Observe that if $j \in \nu(i)$ then there exists $k \in \IndexingSet$ such that $j \in \sigma(k)$ and there is a path from $i$ to $k$ in $D(\sigma)$. But then there is an edge from $k$ to $j$, so there is a path from $i$ to $j$ in $D(\sigma)$.  
\end{proof}
\begin{lemma}
    $\nu$ is an ornamentation.
\end{lemma}
\begin{proof}
First we show that for each $i \in \IndexingSet$, we have $\nu(i) \in \BB$.  Let $$\mathcal P = \{P_j \mid j \in \IndexingSet \text{ and $P_j$ is a path from $i$ to $j$ in $D(\sigma)$}\}.$$

As pointed building sets are transitively closed, for each $P \in \mathcal P$, $\left(\bigcup_{j \in P} \sigma(j),\, i\right) \in \BB$.
We can write $\nu(i)$ as $$\nu(i) = \left(\bigcup\limits_{P \in \mathcal P } \bigcup_{j \in P} \sigma(j), \,\, i\right)$$
which is a union of sets pointed at $i$ and is hence in $\BB$.

        Let $j \in \nu(i)$ and let $k \in \nu(j)$. By \cref{lem:path}, there exists a path from $i$ to $j$ and a path from $j$ to $k$ in $D(\sigma)$. Therefore, there is a path from $i$ to $k$ in $D(\sigma)$. We conclude that $k \in \sigma(k) \subseteq \nu(i)$ and hence $\nu(j) \subseteq \nu(i)$.
\end{proof}

\begin{theorem}
For all $\Omega \subseteq \OO(\BB)$, we have that $\nu = \bigjoin \Omega$.
    
\end{theorem}

\begin{proof}
First observe that by construction, for all $\rho \in \Omega$ we have $\rho \preceq \nu$. Let $\nu' \in \OO(\BB)$ be such that for all $\rho \in \Omega$ we have $\rho \preceq \nu'$. It is immediate that for all $i \in \IndexingSet$, $\sigma(i) \subseteq \nu'(i)$. 

Define a directed graph $D(\nu')$ on $\IndexingSet$ with edge set $\{(i, j) \mid j \in \nu'(i)\}.$
    For $i \in \IndexingSet$ let $$\pi'(i) := \{j \in \IndexingSet \mid \text{there exists a path from $i$ to $j$ in }D(\nu')\}.$$

        Observe that if $(i, j)$ is an edge in $D(\sigma)$ then $(i, j)$ is also an edge in $D(\nu')$. Then it follows that $\pi(i) \subseteq \pi'(i)$ for all $i$. As $\nu'$ is an ornamentation, $\pi'(i) = \nu'(i)$ for all $i$.

    For $i \in \IndexingSet$, we have 
    $$\begin{aligned}
    \nu(i) &= \bigcup\limits_{j \in \pi(i)} \sigma(j)\\
    &\subseteq \bigcup\limits_{j \in \pi'(i)} \sigma(j)\\
    &\subseteq \bigcup\limits_{j \in \pi'(i)} \nu'(j)\\
    &= \nu'(i).
    \end{aligned}$$

    Hence $\nu \preceq \nu'$ and $\nu = \bigjoin \Omega$.
\end{proof}

\subsection{Lattice Properties}

In this section, we discuss several properties of ornamentation lattices. 

\begin{proposition}
    $\OO(\BB)$ is atomic if and only if $\BB|_{i}$ is atomic for each $i \in \IndexingSet$.
\end{proposition}
\begin{proof}
Suppose each $\BB|_i$ is atomic. Let $\rho \in \OO(\BB)$ and for $i \in \IndexingSet$, let $\rho_i : \IndexingSet \to \BB$ be defined by $$\rho_i(j) = \begin{cases} \rho(j) & j = i \\ (\{j\}, j) & j \neq i.\end{cases}$$ Clearly $\rho = \bigvee\limits_{i \in \IndexingSet} \rho_i$, and as each $\BB|_i$ is atomic, each $\rho_i$ is a join of atoms.

Conversely, suppose $\OO(\BB)$ is atomic. Then $\BB|_i$ is atomic as it is an order ideal in $\OO(\BB)$.
\end{proof}

\begin{lemma}
$\rho \in \OO(\BB)$ is join-irreducible if and only if 
\begin{itemize}
    \item $\rho = \rho_{(S,i)}$ for some $(S, i) \in \BB$ and
    \item $(S,i)$ is join-irreducible in $\BB|_{i}$.
\end{itemize}
    
\end{lemma}
\begin{proof}
It is automatic that if $(S,i)$ is join-irreducible in $\BB|_{i}$, then $\rho_{(S,i)}$ is join-irreducible in $\OO(\BB)$. 

Now suppose that $\rho \in \OO(\BB)$ is join-irreducible. First we show that $\rho = \rho_{(S,i)}$ for some $(S,i) \in \BB$. Suppose not. Then there exist $i, j \in \IndexingSet$ such that $|\rho(i)|, |\rho(j)| \ge 2$. Then define $\rho^i \in \OO(\BB)$ by $$\rho^i(k) = \begin{cases} (\{i\}, i) & k = i \\ \rho(k) & k \neq i \end{cases}$$ and similarly for $\rho^j$. Then we have $\rho = \rho^i \join \rho^j$, so $\rho$ is not join-irreducible.

Then $\rho = \rho_{(S,i)}$ for some $(S,i) \in \BB$ and it is automatic that $(S,i)$ must be join-irreducible in $\BB|_{i}$.

\end{proof}

\begin{definition}
    A pointed building set $\BB$ is \emph{acyclic} if whenever $(S, i), (T, j) \in \BB$ such that $i \in T \tand j \in S$ then $i = j$.
\end{definition}

\begin{remark}
    The digraphical pointed building set of an acyclic directed graph is acyclic.
\end{remark}

\begin{lemma}
\label{lem:cover_relations}
Suppose $\BB$ is finite and acyclic, and let \(\rho, \sigma \in \OO(\BB)\) such that \(\rho \precdot \sigma\). Then there exists a unique \(i \in \IndexingSet\) such that \(\rho(i) \subsetneq \sigma(i)\) and for all other \(j \in \IndexingSet, \rho(j) = \sigma(j)\). 
Furthermore, at least one of the following two holds: \begin{itemize} 
\item there exists $(S, i) \in \BB|_i$  such that $\sigma(i) = \left(\bigcup\limits_{j \in S} \rho(j),\ i\right)$
\item \(\rho(i) \precdot \sigma(i)\) in \(\BB|_{i}\). 
\end{itemize} 
\end{lemma}

\begin{proof}
Let $\rho \precdot \sigma$. By definition, there must exist $i \in \IndexingSet$ such that $\rho(i) \subsetneq \sigma(i)$. Now let $$X := \{i \in \IndexingSet \mid \rho(i) \subsetneq \sigma(i)\}.$$ Suppose by way of contradiction that $|X| \ge 2$. Define a partial order on $X$ where $i \preceq_X j$ if $j \in \sigma(i)$, and pick $i \in X$ minimal with respect to $\preceq_X$. We note that $\preceq_X$ is a partial order as $\BB$ is acyclic and has a minimal element as $\IndexingSet$ is finite. 
Then define $\pi \in \OO(\BB)$ by $$\pi(k) = \begin{cases}
    \sigma(k) & k = i \\
    \rho(k) & k \neq i.
\end{cases}$$ 
We claim that $\pi$ is an ornamentation. If this holds, we will have that $\rho \prec \pi \prec \sigma$, contradicting that $\rho \precdot \sigma$. The only non-trivial condition to check is transitive closure when $k \neq i$ and $i \in \pi(k) = \rho(k)$. We claim that $\rho(k) = \sigma(k)$. Otherwise, $k \in X$ and
$i \in \rho(k) \subseteq \sigma(k)$, so $k \preceq_X i$.
By minimality of $i \in X$, this forces $k \notin X$, a contradiction, so $\rho(k) = \sigma(k)$.
Since $\sigma$ is an ornamentation and $i \in \sigma(k)$, we have
\[
\sigma(i) \subseteq \sigma(k) = \rho(k) = \pi(k),
\]
so the closure condition holds.

Now suppose that $\rho(i) \not\precdot \sigma(i)$ in $\BB|_i$. Then there exists $(S, i) \in \BB|_i$ such that $\rho(i) \subsetneq (S,i) \subsetneq \sigma(i)$. Let $$(T,i) = \left(\bigcup\limits_{j \in S} \rho(j),\ i\right).$$ Observe that for each $j \in S$, we have $\left(\rho(j) \cup S, i\right) \in \BB|_{i}$, so $$(T,i) = \bigcup\limits_{j \in S} (\rho(j) \cup (S, i)) \in \BB|_i.$$
Define $\tau: \IndexingSet \to \BB$ by $$\tau(j) = \begin{cases}
    (T,i) & j = i \\
    \rho(j) & j \neq i.
\end{cases}$$

As $(S,i) \subseteq \sigma(i)$ and for each $j \in S, \rho(j) \subseteq \sigma(i)$ we have $(T,i) \subseteq \sigma(i)$. We claim that $\tau \in \OO(\BB)$. It suffices to check that 
\begin{itemize}
    \item if $j \in \IndexingSet - \{i\}$ and $i \in \rho(j)$, then $(T,i) \subseteq \rho(j)$ and
    \item if $j \in \IndexingSet$ and $j \in (T,i)$, then $\rho(j) \subseteq (T,i)$.
\end{itemize}

In the first case, suppose that $j \in \IndexingSet - \{i\}$ and $i \in \rho(j)$. Then $$(T,i) \subseteq \sigma(i)  \subseteq \sigma(j) = \rho(j).$$ 
In the second case, suppose that $j \in (T,i)$. Then there exists $k \in S$ such that $j \in \rho(k)$. Hence $\rho(j) \subseteq \rho(k) \subseteq (T,i)$.

Hence $\rho \prec \tau \preceq \sigma$ and as $\rho \precdot \sigma$ we find that $\tau = \sigma$.

\end{proof}

\begin{remark}
    If $\BB$ is not assumed to be acyclic, then the previous lemma does not hold. One may observe this in the graphical pointed building set of $K_3$. The ornamentation 
    $$\begin{aligned}
    \rho(1) &= (\{1, 2\}, 1), 
    \\\rho(2) &= (\{1, 2\}, 2), 
    \\\rho(3) &= (\{1, 2, 3\}, 3)
    \end{aligned}$$ is covered by 
    $$\begin{aligned} \sigma(1) &= (\{1, 2, 3\}, 1), 
    \\ \sigma(2) &= (\{1, 2, 3\}, 2), 
    \\ \sigma(3) &= (\{1, 2, 3\}, 3).\end{aligned}$$
\end{remark}

The following lemma refines the \cref{lem:cover_relations} for a large class of natural pointed building sets including both graphical and digraphical pointed building sets.

\begin{lemma} \label{lem:covering_relation_when_nice_building_set}

Suppose $\BB$ is a finite, acyclic pointed building set over $\IndexingSet$ such that for all $i \in \IndexingSet$, if $(S,i) \precdot (T,i)$ in $\BB|_i$ then $|T-S| = 1$. Let $\rho \precdot \sigma$ in $\OO(\BB)$ and let $i \in \IndexingSet$ such that $\rho(i) \subsetneq \sigma(i)$. Then there exists $j \in \IndexingSet$ such that $\sigma(i) = \rho(i) \cup \rho(j)$.
\end{lemma}
\begin{proof}
Let $\rho \precdot \sigma$ and let $i \in \IndexingSet$ such that $\rho(i) \subsetneq \sigma(i)$. By the assumptions, there exists $(S,i) \in \BB|_i$ such that $\rho(i) \subsetneq (S,i) \subseteq \sigma(i)$ and $S = \rho(i) \cup \{j\}$ for some $j \in \IndexingSet$. Define $\tau: \IndexingSet \to \BB$ by $$\tau(k) = \begin{cases}
    \rho(k) & k \neq i \\
    (\rho(i) \cup \rho(j),\, i) & k = i
\end{cases}$$

We claim that $\tau \in \OO(\BB)$. First observe that $(\rho(i) \cup \rho(j), i) = (S \cup \rho(j), i) \in \BB|_i$ by (2) of \cref{def:pointed_building_set}. It suffices to show \begin{itemize}
    \item if $k \in \IndexingSet - \{i\}$ and $i \in \tau(k)$, then $\tau(i) \subseteq \tau(k)$ and
    \item if $k \in \IndexingSet-\{i\}$ and $k \in \tau(i)$, then $\tau(k) \subseteq \tau(i)$.
\end{itemize}
Observe that $\tau(i) \subseteq \sigma(i)$ as $j \in S \subseteq \sigma(i)$.

In the first case, if $i \in \tau(k)$ then $$\tau(i) \subseteq \sigma(i) \subseteq \sigma(k) = \rho(k) = \tau(k).$$ 
In the second case, if $k \in \tau(i)$, then $k \in \rho(i)$ (and $\rho(k) \subseteq \rho(i)$) or $k \in \rho(j)$ (and $\rho(k)\subseteq \rho(j)$). Therefore $\tau(k) = \rho(k) \subseteq \tau(i).$

Hence $\rho \prec \tau \preceq \sigma$ and as $\rho \precdot \sigma$, we conclude that $\tau = \sigma$.

\end{proof}

We will make heavy use of the following classical result.
\begin{proposition}[{\cite{Free}*{Theorem 2.56}}]\label{prop:SemiDistributiveLemma}
Let $L$ be a finite lattice. Then:
\begin{itemize}
\item $L$ is join-semidistributive if and only if for every cover relation $x\precdot y$ in $L$, the set $\Delta_L(y)-\Delta_L(x)$ has a unique minimal element. 
\item $L$ is meet-semidistributive if and only if for every cover relation $x\precdot y$ in $L$, the set $\nabla_L(x)-\nabla_L(y)$ has a unique maximal element. 
\end{itemize}
\end{proposition}

\begin{theorem}
\label{thm:chain_implies_semidistributive}
    Suppose that $\BB$ is a finite, acyclic pointed building set and that for all $i \in \IndexingSet$, $\BB|_i$ is a chain. Then $\OO(\BB)$ is semidistributive.
\end{theorem}
\begin{proof} 
Let \(\rho \precdot \sigma\) and let \(i \in \IndexingSet\) such that \(\rho(i) \subsetneq \sigma(i)\). We show that \begin{enumerate} 
    \item \(\Delta_{\OO(\BB)}(\sigma)-\Delta_{\OO(\BB)}(\rho)\) has a unique minimal element 
    \item \(\nabla_{\OO(\BB)}(\rho)-\nabla_{\OO(\BB)}(\sigma)\) has a unique maximal element. 
\end{enumerate}

Define \(\mu_\downarrow: \IndexingSet \to \BB \tand \mu_\uparrow: \IndexingSet \to \BB\) by 

$$
\begin{aligned}
\mu_\downarrow(j) &= 
    \begin{cases} 
        (\{j\}, j) & j \neq i \\
        \min\{S \in \BB|_{i} : \rho(i) \subsetneq S\} & j = i \end{cases}
\\\\ \mu_\uparrow(j) &= 
    \begin{cases} 
        \max \BB|_j & j \neq i 
        \\ \max\{S \in \BB|_{i} : S \subsetneq \sigma(i)\} & j = i. \end{cases}
\end{aligned}
$$

Both $\{S \in \BB|_{i} : \rho(i) \subsetneq S\}$ and $\{S \in \BB|_{i} : S \subsetneq \sigma(i)\}$ are non-empty as $\sigma(i)$ is in the first set and $\rho(i)$ is in the second respectively.
As each $\BB|_i$ is a chain, each of these sets has a unique minimum and unique maximum. Hence $\mu_\downarrow$ and $\mu_\uparrow$ are both well-defined. Furthermore, each is an ornamentation. 

The following are straightforward to verify:
\begin{itemize}
    \item  $\mu_\downarrow \preceq \sigma$,
    \item $\mu_\downarrow \not\preceq \rho$,
    \item if $\mu \preceq \sigma$ and $\mu \not\preceq \rho$ then $\mu_\downarrow \preceq \mu$,
    \item $\rho \preceq \mu_\uparrow$,
    \item $\sigma \not\preceq \mu_\uparrow$,
    \item if $\rho \preceq \mu$ and $\sigma \not\preceq \mu$ then $\mu \preceq \mu_\uparrow$.
\end{itemize}

\end{proof}

\begin{corollary}
\label{cor:rooted_tree_towards_root_semidistributive}
Let $\TT$ be a rooted tree with edges directed towards the root and let $\BB(\TT)$ be the digraphical pointed building set of $\TT$. Then $\OO(\BB(\TT))$ is semidistributive.
\end{corollary}

\section{Digraphical ornamentation lattices}
\label{sec:Digraphical_Ornamentation_Lattices}

In this section, we give special attention to digraphical ornamentation lattices with a special focus on directed trees.

\begin{remark}
\label{rem:MacNeille_completion}
    In~\cite{abram2025ornamentation}, it is shown that if $D$ is an increasing directed tree, then $\OO(\BB(D))$ is isomorphic to the Dedekind-MacNeille completion of the hypergraphic poset where the underlying hypergraph consists of all paths in $D$.
\end{remark}

\begin{lemma}
\label{lem:covering_relation_DAG}
    Let $D = (V,E)$ be a directed acyclic graph and $\BB(D)$ the digraphical pointed building set of $D$. Then if $\rho \precdot \sigma$ in $\OO(\BB(D))$, there exist unique vertices $u, v \in V$ such that $\sigma(u) = \rho(u) \cup \rho(v)$.
\end{lemma}
\begin{proof}
    By \cref{lem:covering_relation_when_nice_building_set}, there exists a unique $u \in V$ and possibly not unique $v \in V$ such that $\sigma(u) = \rho(u) \cup \rho(v)$. We claim that $v$ is the unique vertex with this property. Indeed, observe that for all $w \in \sigma(u) - \rho(u)$, there is a path from $v$ to $w$ in $D$. Furthermore as $v \notin \rho(u)$, $v \in \sigma(u) - \rho(u)$. Then $v$ is the unique minimum element of $\sigma(u)-\rho(u)$, viewed as a poset.
\end{proof}

The following is the main result of this section: 
\begin{theorem}[Directed tree duality]
\label{thm:tree_duality}
    Let $D$ be a directed tree. Then $\OO(\BB(D^\opp)) \simeq \OO(\BB(D))^\opp$. 
\end{theorem}

\begin{figure}
    \centering
    \includegraphics[height=.47\textheight]{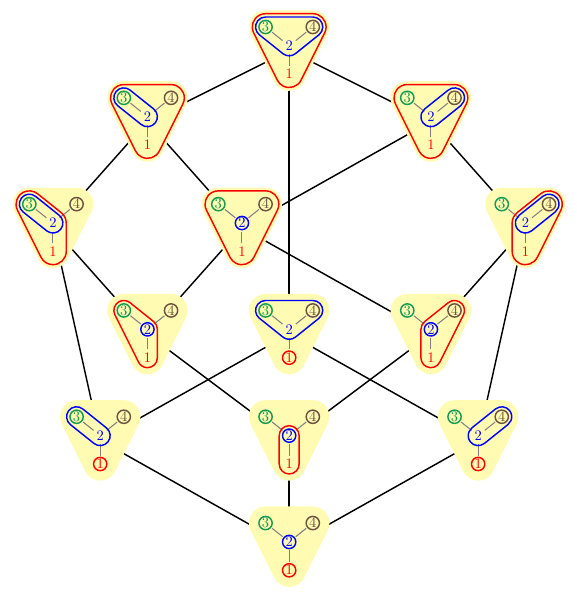}  \\ 
    \includegraphics[height=0.47\textheight]{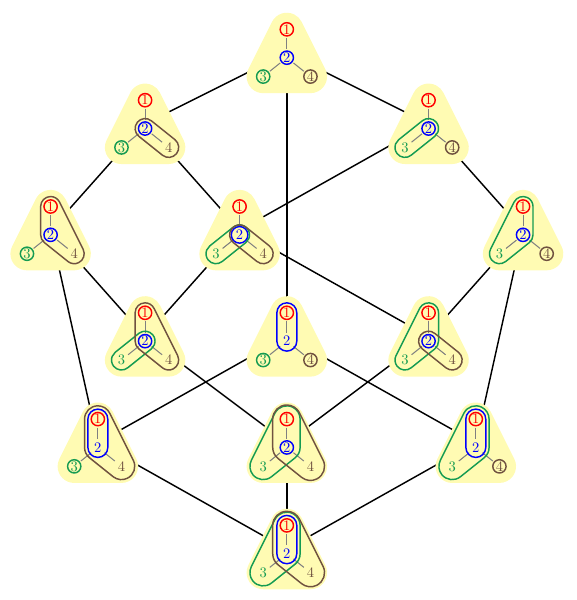}
    \caption{Illustration of \cref{thm:tree_duality}. Top: $\OO(\BB(D))$. Bottom: $\OO(\BB(D^\opp))^\opp$. Each tree is oriented upwards. Bottom is a relabelling of~\cref{fig:ornamentations_example}.}
    \label{fig:TreeDualityLattices}
\end{figure}

\cref{thm:tree_duality} appears deceptively simple. For example, it fails for directed acyclic graphs as seen in \cref{fig:TreeDualityFail}. We will prove this theorem by providing an explicit construction of the isomorphism. However, one can obtain a geometric proof using~\cref{rem:MacNeille_completion} by reversing the direction of orientation and applying an appropriate relabeling.

\begin{figure}
    \centering
    \includegraphics[width=0.5\linewidth]{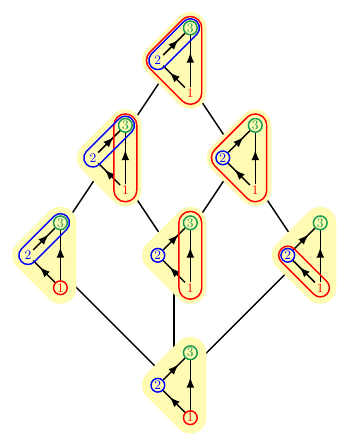}
    \caption{Failure of \cref{thm:tree_duality} for a directed acyclic graph $D$. Observe that $D^\opp \simeq D$ but $\OO(\BB(D))^\opp \not\simeq \OO(\BB(D^\opp))$.}
    \label{fig:TreeDualityFail}
\end{figure}

\begin{definition}
\label{def:tree_flip_map}
 Let $D = (V,E)$ be a directed tree. We define a map $\Psi_D: \OO(\BB(D)) \to \OO(\BB(D^\opp))$.

    Let $\rho \in \OO(\BB(D))$ and $u \in V$. Define $$\omega_D(\rho)(u) := \{v \in V \mid v \preceq_D u \tand u \notin \rho(v)\} \cup \{u\}.$$
    
    Let $$\Psi_D(\rho)(u) := (\max\{S \in \BB(D^\opp)|_u \mid S \subseteq \omega_D(\rho)(u)\},\, u).$$ 

    Equivalently,  $v \in \Psi_D(\rho)(u)$ if and only if: \begin{itemize}
        \item there is a (unique) path from $u$ to $v$ in $D^\opp$ and
        \item for all vertices $w$ in the half-open interval $[v, u)$ in $D$, $u \notin \rho(w)$.
            \end{itemize}
    
\end{definition}

See \cref{fig:tree_flip_map} for a more complex example of $\Psi_D$ than seen in \cref{fig:TreeDualityLattices}.

\begin{figure}
    \centering
    \includegraphics[width=0.4\linewidth]{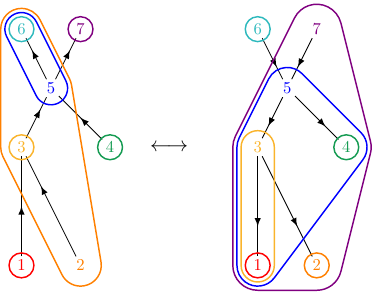}
    \caption{A more involved example of \cref{def:tree_flip_map}.}
    \label{fig:tree_flip_map}
\end{figure}

\begin{lemma}
\label{lem:tree_duality_well_defined}
For all $\rho \in \OO(\BB(D))$, $\Psi_D(\rho)$ is an ornamentation.
\end{lemma}
\begin{proof}

    First observe that for all $u \in V$, $\Psi_D(\rho)(u) \in \BB(D^\opp)|_u$ by construction.
    Now let $u \in V$, $v \in \Psi_D(\rho)(u)$, and $w \in \Psi_D(\rho)(v)$. 

    There is a unique path from $w$ to $v$ and a unique path from $v$ to $u$ in $D$.
    Let $x \in [w, u)$ be a vertex on the unique path from $w$ to $u$. We claim that $u \notin \rho(x)$. 

    One of the following holds:
    \begin{itemize}
        \item $x \in [v, u)$ in $D$ or
        \item $x \in [w, v)$ in $D$.
    \end{itemize}
    In the first case, $u \notin \rho(x)$ as $v \in \Psi_D(\rho)(u)$. In the second case, $v \notin \rho(x)$. As $v$ is on the unique path from $x$ to $u$ in $D$, $u \notin \rho(x)$ as well. (As if $u \in \rho(x)$, the entire path from $x$ to $u$ must be contained in $\rho(x)$.)

    Hence $u \notin \rho(x)$, so $w \in \Psi_D(\rho)(u)$.
    \end{proof}

\begin{lemma}
\label{lem:tree_duality_order_reversing}
    $\Psi_D$ is order-reversing.
\end{lemma}
\begin{proof}
    Let $\rho, \sigma \in \OO(\BB(D))$ with $\rho \preceq \sigma$ and $u,v \in V$ with $v \in \omega_D(\sigma)(u)$. Then either $u = v$ or $u \notin \sigma(v)$ and hence $u \notin \rho(v)$. In either case, $v \in \omega_D(\rho)(u)$. Hence $\omega_D(\sigma)(u) \subseteq \omega_D(\rho)(u)$. Then by the construction of $\Psi_D$, we have that $\Psi_D(\sigma)(u) \subseteq \Psi_D(\rho)(u)$ and so $\Psi_D$ is order-reversing.

\end{proof}

\begin{lemma}
\label{lem:tree_duality_inverse}
    For any directed tree $D$, we have $\Psi_{D^\opp} \circ \Psi_{D} = \id$.

\end{lemma}

\begin{proof}
Let $\rho \in \OO(\BB(D)) \tand u \in V$. First we show that $\rho(u) \subseteq \Psi_{D^\opp}(\Psi_{D}(\rho))(u)$. 

Let $v \in \rho(u)$. If $v = u$ then $u \in \Psi_{D^\opp}(\Psi_{D}(\rho))(u)$ as $\Psi_{D^\opp}(\Psi_{D}(\rho))$ is an ornamentation. Now we assume that $v \neq u$. Then for all vertices $w \in (u, v]$ in $D$,  we have $w \in \rho(u)$ and so $u \notin \Psi_D(\rho)(w)$. Therefore $v \in \Psi_{D^\opp}(\Psi_{D}(\rho))(u)$, and so $\rho(u) \subseteq \Psi_{D^\opp}(\Psi_{D}(\rho))(u)$.

It remains to show that $\Psi_{D^\opp}(\Psi_{D}(\rho))(u) \subseteq \rho(u)$.  Let $v \in \Psi_{D^\opp}(\Psi_{D}(\rho))(u)$ and let $w \in [u,v]$ be the unique minimal vertex (in $D$) such that $v \in \rho(w)$. One may observe that $w$ exists as $v \in \rho(v)$. Suppose by way of contradiction that $w \neq u$. 

As $v \in \Psi_{D^\opp}(\Psi_{D}(\rho))(u)$ and $w \in (u, v]$, we have that $u \notin \Psi_{D}(\rho)(w)$. Then there exists $x \in [u, v)$ in $D$ such that $w \in \rho(x)$. Hence $v \in \rho(w) \subseteq \rho(x)$, contradicting the minimality of $w$. Therefore $v \in \rho(u)$ and so $\Psi_{D^\opp}(\Psi_{D}(\rho))(u) \subseteq \rho(u)$.
\end{proof}

\begin{proof}[Proof of \cref{thm:tree_duality}]

\cref{lem:tree_duality_well_defined} shows that $\Psi_D$ is well-defined. \cref{lem:tree_duality_order_reversing} shows that $\Psi_D$ is an order-preserving map to $\OO(\BB(D^\opp))^\opp$. \cref{lem:tree_duality_inverse} shows that $\Psi_D$ is a bijection and its inverse, $\Psi_{D^\opp}: \OO(\BB(D^\opp))^\opp \to \OO(\BB(D))$, is also order-preserving by \cref{lem:tree_duality_order_reversing}. Hence $\OO(\BB(D^\opp)) \simeq \OO(\BB(D))^\opp$.
    
\end{proof}

\begin{corollary}
\label{cor:rooted_tree_away_from_root_semidistributive}
Let $D$ be a finite rooted tree with edges directed \emph{away} from the root. Then $\OO(\BB(D))$ is semidistributive.
\end{corollary}
\begin{proof}
This follows directly from \cref{cor:rooted_tree_towards_root_semidistributive} and \cref{thm:tree_duality}.
\end{proof}

\cref{cor:rooted_tree_away_from_root_semidistributive} was first proven in~\cite{ajran2025pop}. We can strengthen this result to all finite directed trees, following essentially the same proof from~\cite{ajran2025pop}.
\begin{theorem}
    If $D = (V,E)$ is a finite directed tree, then $\OO(\BB(D))$ is semidistributive.
\end{theorem}
\begin{proof}
We will again make use of \cref{prop:SemiDistributiveLemma}. Let $\rho \precdot \sigma$ in $\OO(\BB(D))$. By \cref{lem:covering_relation_DAG}, there exist unique $u, v \in V$ such that $\sigma(u) = \rho(u) \cup \rho(v)$.

As $D$ is a directed tree and $v \in \sigma(u)$, there is a unique path from $u$ to $v$ in $D$, which is the interval $[u,v]$. 

We show that 
\begin{enumerate} 
    \item \(\Delta_{\OO(\BB)}(\sigma)-\Delta_{\OO(\BB)}(\rho)\) has a unique minimal element 
    \item \(\nabla_{\OO(\BB)}(\rho)-\nabla_{\OO(\BB)}(\sigma)\) has a unique maximal element. 
\end{enumerate}

Define \(\mu_\downarrow: \IndexingSet \to \BB \tand \mu_\uparrow: \IndexingSet \to \BB\) by 

$$
\begin{aligned}
\mu_\downarrow(w) &= 
    \begin{cases} 
        ([u,v], u) & w = u \\
        (\{w\}, w) & w \neq u 

    \end{cases}\\
\\ \mu_\uparrow(w) &= 
    \begin{cases} 
        (\nabla_D(w)-\nabla_D(v),\, w) & u \preceq_D w \prec_D v\\
        (\nabla_D(w),\, w) & \text{otherwise.} 
    \end{cases}
\end{aligned}
$$
It is straightforward to verify that $$\mu_\downarrow \in \Delta_{\OO(\BB)}(\sigma)-\Delta_{\OO(\BB)}(\rho) \tand \mu_\uparrow \in \nabla_{\OO(\BB)}(\rho)-\nabla_{\OO(\BB)}(\sigma).$$ 

Suppose that $\tau \in \Delta_{\OO(\BB)}(\sigma)-\Delta_{\OO(\BB)}(\rho)$. Then $\rho(u)\subsetneq \tau(u) \subseteq \sigma(u)$. Hence there must exist $ w \in (\sigma(u)-\rho(u)) \cap \tau(u)$. As $v$ is on the unique path from $u$ to $w$, $v \in \tau(u)$. Hence $\mu_\downarrow \preceq \tau$.

Now suppose that $\tau \in \nabla_{\OO(\BB)}(\rho)-\nabla_{\OO(\BB)}(\sigma)$. Then $v \notin \tau(u)$ as otherwise $$\sigma(v) = \rho(v) \subseteq \tau(v) \subseteq  \tau(u)$$ and $\tau \preceq \sigma.$ Then for $w \in \nabla_D(v)$, $w \notin \tau(u)$ as the unique path from $u$ to $w$ contains $v$. A similar argument holds, for any $t$ such that $u \preceq_D t \prec_D v$. Hence $\tau \preceq \mu_\uparrow$.
\end{proof}

\section{Direct limits of pointed building sets induce inverse limits of ornamentation lattices}
\label{sec:direct_limits}

    Let $\DirectedSet$ be a directed set, let $\{(\IndexingSet_\alpha, \BB_\alpha)\}_{\alpha \in \DirectedSet}$ be a collection of pairs such that
    \begin{itemize}
        \item for all $\alpha \in \DirectedSet$, $\IndexingSet_\alpha$ is an indexing set and $\BB_\alpha$ is a pointed building set over $\IndexingSet_\alpha$;
        \item for all $\alpha \le \beta$ in $\DirectedSet$, we have $\IndexingSet_\alpha \subseteq \IndexingSet_\beta$ and $\BB_\alpha \subseteq \BB_\beta$;
        \item for all $\alpha \in \DirectedSet$, $\BB_\alpha$ is \dfn{locally finite}, i.e., for all $i \in \IndexingSet_\alpha$, all sets in $\BB|_i$ are finite.
    \end{itemize}

\begin{definition}
We define $$\DirectLimit \IndexingSet_\alpha := \bigcup\limits_{\alpha \in \DirectedSet} \IndexingSet_\alpha.$$ 

For a collection of pointed subsets $\mathcal C$ of $\IndexingSet$, define the \dfn{pointwise union closure} $$\PointwiseUnionClosure(\mathcal C) := \left\{ \left(\bigcup_{(S, i) \in \AA} S,\ i \right) \mid i \in \IndexingSet \tand \emptyset \subsetneq \AA \subseteq \mathcal C|_i\right\}.$$ 

By an abuse of notation, we define $\DirectLimit \BB_\alpha := \PointwiseUnionClosure\left(\bigcup\limits_{\alpha \in \DirectedSet} \BB_\alpha\right)$.
    
\end{definition}

\begin{remark}
    It would be interesting to make this notation rigorous by defining a category of pointed building sets.
\end{remark}

\begin{lemma}
    $\DirectLimit \BB_\alpha$ is a pointed-building set over $\DirectLimit \IndexingSet_\alpha$.
\end{lemma}
\begin{proof}
We check the three axioms from \cref{def:pointed_building_set}.
\begin{enumerate}
    \item Let $i \in \DirectLimit \IndexingSet_\alpha$. Then there exists $\alpha \in \DirectedSet$ such that $i \in \IndexingSet_\alpha$. Hence $(\{i\}, i) \in \BB_\alpha \subseteq \DirectLimit \BB_\alpha$.
    \item Suppose $(S, i), (T,j) \in \DirectLimit \BB_\alpha$ such that $j \in S$. Further suppose that $k \in T$. Then there exists $\alpha \in \DirectedSet$ and $(S_\alpha, i) \in \BB_\alpha$ such that $j \in S_\alpha$ and $S_\alpha \subseteq S$ and there exists $\beta \in \DirectedSet$ and $(T_\beta, j) \in \BB_\beta$ such that $k \in T_\beta$ and $T_\beta \subseteq T$. There exists $\gamma \in \DirectedSet$ such that $\alpha, \beta \preceq \gamma$. Hence $(S_\alpha, i), (T_\beta, j) \in \BB_\gamma$ and as $\BB_\gamma$ is a pointed building set, $(S_\alpha \cup T_\beta, i) \in \BB_\gamma$. Hence $(S_\alpha \cup T_\beta, i) \in \DirectLimit \BB_\alpha$. 
    Let $S_k = S_\alpha \cup T_\beta$. Observe that $S \cup T = S \cup \left( \bigcup_{k \in T} S_k \right)$ and as $\DirectLimit \BB_\alpha|_i$ is closed under arbitrary (non-empty) unions, $(S \cup T, i) \in \DirectLimit \BB_\alpha$.

    \item By construction, for all $i \in \DirectLimit \IndexingSet_\alpha,\  \DirectLimit \BB_\alpha|_i$ is closed under arbitrary unions.
\end{enumerate}

\end{proof}

\begin{definition}
    For $\alpha \le \beta$ in $\DirectedSet$, define $\Psi_{\beta\alpha}: \OO(\BB_\beta)\to \OO(\BB_\alpha)$ by $$\Psi_{\beta\alpha}(\rho)(i) = \left(\max\{S \in \BB_\alpha|_{i} \mid S \subseteq \rho(i)\},\ i \right).$$

    For $\alpha \in \DirectedSet$, we also define projection maps $\pi_\alpha: \OO(\DirectLimit \BB_\beta) \to \OO(\BB_\alpha)$ by 
    $$\pi_\alpha(\rho)(i) = \left(\max\{S \in \BB_\alpha|_{i} \mid S \subseteq \rho(i)\},\ i \right).$$
\end{definition}

\begin{lemma}
For all $\alpha, \beta \in \DirectedSet$, both $\Psi_{\beta\alpha}$ and $\pi_\alpha$ are well-defined. 
\end{lemma}
\begin{proof}
The same proof holds for both maps.
     For all $i \in \IndexingSet_\alpha$, $\Psi(\rho)(i) \in \BB_\alpha|_{i}$ so it remains to show that if $j \in \Psi(\rho)(i)$ then $\Psi(\rho)(j) \subseteq \Psi(\rho)(i)$. Indeed, if $j \in \Psi(\rho)(i)$ then $j \in \rho(i)$ and $\left(\Psi(\rho)(i) \cup \Psi(\rho)(j),\ i\right) \in \BB_\alpha$. As $j \in \rho(i)$, $\rho(j) \subseteq \rho(i)$, and hence $\Psi(\rho)(j) \subseteq \rho(i)$. Then $\Psi(\rho)(j) \subseteq \Psi(\rho)(i)$ by construction. 

\end{proof}

By construction, $\Psi_{\beta\alpha}$ and $\pi_\alpha$ are order-preserving and for $\alpha \le\beta \le \gamma$ in $\DirectedSet$, $\Psi_{\beta\alpha}\circ \Psi_{\gamma\beta} = \Psi_{\gamma\alpha}$.  Furthermore, $\Psi_{\beta\alpha}\circ\pi_\beta = \pi_\alpha$ This makes $(\OO(\BB_\alpha))_{\alpha \in \DirectedSet}$ into an inverse system. 

\begin{remark}
In general,  $\Psi_{\beta\alpha}$ may not be a lattice map. Let $\mathcal I_1 = \mathcal I_2 = \{1, 2, 3\}$ and let 
    $$\begin{aligned}
    \BB_1 &= \{(\{1\}, 1) , (\{1, 2, 3\}, 1), (\{2\}, 2), (\{3\}, 3)\}\\
   \BB_2 &= \{(\{1\}, 1), (\{1, 2\}, 1), (\{1, 2, 3\}, 1), (\{2\}, 2), (\{2, 3\}, 2), (\{3\}, 3)\} 
    \end{aligned}$$
    $$\begin{aligned} 
    \sigma(1) &= (\{1, 2\}, 1)  \\
    \sigma(2) &= (\{2\}, 2)     \\
    \sigma(3) &= (\{3\}, 3)     \\
    \rho(1) &= (\{1\}, 1)  \\
    \rho(2) &= (\{2, 3\}, 2)     \\
    \rho(3) &= (\{3\}, 3).
    \end{aligned}$$
    Then $\Psi_{21}(\sigma \join \rho)(1) = (\{1, 2, 3\}, 1)$ but $(\Psi_{21}(\sigma) \join \Psi_{21}(\rho))(1) = (\{1\}, 1)$.
It would be interesting to classify when $\Psi_{\beta\alpha}$ is a lattice map.
\end{remark}

We now arrive at the main theorem of this section.

\begin{theorem}
\label{thm:inverse_limit_general}
    In the category of posets, $\InverseLimit \OO(\BB_\alpha) \simeq \OO\left(\DirectLimit \BB_\alpha\right)$.
\end{theorem}
\begin{proof}
We show that $\OO\left(\DirectLimit \BB_\alpha\right)$ satisfies the universal property of an inverse limit. Let $P$ be a poset and  $(\phi_\alpha : P \to \OO(\BB_\alpha))_{\alpha\in\DirectedSet}$ be a family of order-preserving maps such that
$$\text{for all } \beta\ge\alpha\text{ in }\DirectedSet,\quad \Psi_{\beta\alpha} \circ \phi_\beta = \phi_\alpha.$$
We show there is a unique order-preserving map
$$\Phi: P \to \OO\left(\DirectLimit \BB_\alpha\right)$$
such that for all $\alpha \in \DirectedSet$,$$
\pi_\alpha \circ \Phi = \phi_\alpha.$$

Fix $x\in P$. For each $i \in \DirectLimit \IndexingSet_\alpha$ and $\alpha \in \DirectedSet$ such that $i \in \IndexingSet_\alpha$, define
$$\Phi(x)(i) := \left( \bigcup_{\alpha \st i\in\IndexingSet_\alpha} \phi_\alpha(x)(i),\ i \right).$$

We check that $\Phi$ is well-defined. As $\DirectLimit\BB_\alpha$ is closed under arbitrary pointwise unions, $\Phi(x)$ is indeed an element of $\OO(\DirectLimit\BB_\alpha)$. Now suppose that $j \in \Phi(x)(i)$ and $k \in \Phi(x)(j)$. Then there exists $\alpha, \beta, \gamma \in \DirectedSet$ such that 
\begin{enumerate}
    \item $j \in \phi_\alpha(x)(i)$ and
    \item $k \in \phi_\beta(x)(j)$ and
    \item $\alpha, \beta \le \gamma$.
\end{enumerate}
As $k \in \Psi_{\gamma\beta}(\phi_\gamma(x))(j)$, we have $$k \in \phi_\gamma(x)(j) \subseteq \phi_\gamma(x)(i) \subseteq \Phi(x)(i)$$ so $\Phi(x) \in \OO\left(\DirectLimit \BB_\alpha\right)$.

Now we show that $\Phi$ is order-preserving.
    Suppose $x \le y$ in $P$. Then for each $\alpha$ and each $i \in \IndexingSet_\alpha$,
    $$\phi_\alpha(x)(i) \subseteq  \phi_\alpha(y)(i),$$ and taking unions gives $\Phi(x)(i) \subseteq \Phi(y)(i)$. Thus, $\Phi(x) \preceq \Phi(y)$.

    Next we show that $\pi_\alpha\circ\Phi = \phi_\alpha$. 
    By construction, for $i \in \IndexingSet_\alpha, \phi_\alpha(x)(i) \subseteq \Phi(x)(i)$. Hence $\phi_\alpha(x)(i) \subseteq \pi_\alpha(\Phi(x))(i)$. Now suppose that $j \in \pi_\alpha(\Phi(x))(i)$. Then 
    \begin{itemize}
        \item $j \in \Phi(x)(i)$
        \item there exists $(S, i) \in \BB_\alpha$ with $j \in S$ and $S \subseteq \Phi(x)(i)$
        \item for all $k \in S$, there exists $\beta_k \in \DirectedSet$ with $k \in \phi_{\beta_k}(x)(i)$.
    \end{itemize} 
    As $S$ is finite, there exists $\gamma \in \DirectedSet$ such that for all $k \in S$, $\beta_k \le \gamma$. As $$k \in \Psi_{\gamma\beta_k}(\phi_\gamma(x))(i) = \phi_{\beta_k}(x)(i),$$ we have $k \in \phi_\gamma(x)(i)$ as well. Hence $S \subseteq \phi_\gamma(x)(i)$. Hence $$j \in (S, i) \subseteq \Psi_{\gamma\alpha}(\phi_\gamma(x))(i) = \phi_\alpha(x)(i).$$
     
   Finally, we show that $\Phi$ is unique.
    Suppose there were another map $\Phi': P \to \OO\left(\DirectLimit \BB_\alpha\right)$ such that $\pi_\alpha \circ \Phi' = \phi_\alpha$ for all $\alpha$.
    Then for each $x \in P, \alpha \in \DirectedSet$, and $i \in \IndexingSet_\alpha$ we must have
    $$
    \pi_\alpha(\Phi'(x))(i) = \phi_\alpha(x)(i) = \pi_\alpha(\Phi(x))(i).
    $$

    Then $\phi_\alpha(x)(i) \subseteq \Phi'(x)(i)$ and hence $\Phi(x)(i) \subseteq \Phi'(x)(i)$. 
    
    Conversely, if $j \in \Phi'(x)(i)$ then there exists $\alpha \in \DirectedSet$ and $(S, i) \in \BB_\alpha$ such that $j \in S$ and $S \subseteq \Phi'(x)(i)$. Then $$j \in S \subseteq \pi_\alpha(\Phi'(x)) = \phi_\alpha(x)(i) \subseteq \Phi(x)(i).$$
    Hence for all $x \in P$, we have $\Phi(x)(i) = \Phi'(x)(i)$ so $\Phi = \Phi'$.
    
Hence, $\OO\left(\DirectLimit \BB_\alpha\right)$ satisfies the universal property of the inverse limit.

\end{proof}

\begin{corollary}
    Let $[n]$ have its usual total order. For $m \le n$, define $\pi_{nm}:[n] \to [m]$ by $$\pi_{nm}(i) = \begin{cases} i & i \le m \\ m & i > m.\end{cases}$$
    Then $\InverseLimit [n] \simeq \omega +1$, the positive integers with an upper bound adjoined.
\end{corollary}
\begin{proof}
Let $\IndexingSet_n = [n]$ and let $$\BB_n := \{([i], 1) \mid i \in [n]\} \cup \{(\{i\}, i) \mid i \in [n]\}.$$ Then $\OO(\BB_n) \simeq [n]$ and the result follows from \cref{thm:inverse_limit_general}.
\end{proof}

\begin{corollary}
    Let $\mathrm{Bool}_n$ be the $n$-th Boolean lattice, i.e., the subsets of $[n]$ ordered by inclusion. For $m \le n$, define $\pi_{nm}: \mathrm{Bool}_n \to \mathrm{Bool}_m$ by $$\pi_{nm}(S) = S \cap [m].$$
    Then $\InverseLimit \mathrm{Bool}_n \simeq \mathrm{Bool}_{\Z^+}$ the set of subsets of $\Z^+$ ordered by inclusion.
\end{corollary}
\begin{proof}
Let $\IndexingSet_n = \{0, \dots, n\}$ and let $$\BB_n := \{(S, 0) \mid \{0\} \subseteq S \subseteq \{0, \dots, n\}\} \cup \{(\{i\},\; i) \mid i \in [n]\}.$$ Then $\OO(\BB_n) \simeq \mathrm{Bool}_n$ and the result follows from \cref{thm:inverse_limit_general}.
\end{proof}

\section{Infinite Tamari lattices and 312-avoiding total orders}
\label{sec:infinite_Tamari}

Our goal in this section is to study the infinite and bi-infinite Tamari lattices, their generalizations, and connections to 312-avoiding total orders. We will also discuss the continuous Tamari lattice. First we review some properties of total orders. For this section, let $X$ be a set that is totally ordered under a (fixed) relation $\le$.

\subsection{The weak order on total orders}

\begin{definition}
Let $\preceq$ be a (possibly different) total order on $X$.

    The set of \dfn{inversions} of $\preceq$ relative to $\le$ is defined as $$\Inv(\preceq) := \{ (i, j) \in X \times X \mid i < j \tand j \prec i\}.$$
    
    We define $\Weak(X_\le)$ to be the set of total orders on $X$ ordered under inclusion of inversion sets relative to $\le$.

\end{definition}

Observe that $\preceq$ is recoverable from $\Inv(\preceq)$. When $X_\le = [n]$ under their usual ordering, $\Weak(X_\le) \simeq \Weak(S_n)$, the weak order on permutations.

\begin{proposition}
\label{prop:weak_order}
    $\Weak(X_\le)$ is a complete lattice.
\end{proposition}

This proposition is well-known in the finite case. In the case that $X_\le = \Z_\le$, the integers under their usual ordering, it is proven in~\cite{barkley2022combinatorial}, that $\Weak(\Z_\le)$ is a lattice. We extend their proof to all total orders. We recall the following characterization of inversion sets of total orders.

\begin{lemma}
    A subset $S \subseteq \Inv(\le^\opp)$ is the inversion set of a total order if and only if both $S$ and $\Inv(\le^\opp) - S$ are transitively closed. 
\end{lemma}

\begin{proof}
Note that $ \Inv(\le^\opp) = \{(i, j) \in X \times X \mid i < j\}$.
    Let $\preceq$ be a total order on $X$. If $(a, b) \tand (b, c) \in \Inv(\preceq)$ then $a < b < c \tand c \prec b \prec a$. Hence $(a, c) \in \Inv(\preceq)$. Similarly if $a < b < c$ and $(a, b), (b, c) \notin S$ then $a \prec b \prec c $ so $(a, c) \notin S$. Hence both $S$ and $\Inv(\le^\opp) - S$ are transitively closed. 

    Now suppose that both $S$ and $\Inv(\le^\opp) - S$ are transitively closed and define $\prec$ by $$a \prec b \Leftrightarrow 
        \begin{cases}
         (b, a) \in S \\ 
         (a, b) \notin S \tand a < b.
        \end{cases}\tor $$ 

For $a < b$, exactly one of the following holds:
\begin{itemize}
    \item $(a, b) \in S$ so $b \prec a$
    \item $(a, b) \notin S$ and $a \prec b$.
\end{itemize}

Hence exactly one of $a \prec b$ or $b \prec a$ holds. Now we show that $\prec$ is transitive. Suppose that $a \prec b$ and $b \prec c$. The remainder of the proof reduces to checking the cases of the 6 possible orders of $a, b, \tand c$ under $<$.    
\end{proof}

Note that $\Inv(\le^\opp) - S$ being transitively closed is equivalent to the property that if $(a, c) \in S$ and $a < b < c$ then at least one of $(a, b)$ or $(b, c)$ is in $S$. This property is also known as $S$ being \dfn{transitively coclosed.}

\begin{proof}[Proof of \cref{prop:weak_order}]
First observe that $\Weak(X_\le)$ is bounded below by $\le$ and above by $\le^\opp$. Hence it suffices to show that $\Weak(X_\le)$ is a complete join-semilattice. Now let $\Omega$ be a collection of total orders. Define $$\Inv(\Omega) := \bigcup\limits_{\preceq \in \Omega} \Inv(\preceq)$$ and $\overline{\Inv(\Omega)}$ to be its transitive closure. We claim that $\overline{\Inv(\Omega)} = \Inv(\preceq_\Omega)$ for some total order $\preceq_\Omega$. 

A well-known characterization of the transitive closure of a relation $R$ is as the union $$\overline R = \bigcup_{n = 1}^\infty R^n.$$ In particular, this implies that $(a, b) \in \overline{\Inv(\Omega)}$ if and only if there exists a sequence $a = x_1, x_2, \dots, x_n = b$ such that for all $i$, we have $(x_i, x_{i+1}) \in \Inv(\preceq_i)$ for some $\preceq_i \in \Omega$.

In particular, if $(a, b) \in \overline{\Inv(\Omega)}$ then $a < b$, and so $\overline{\Inv(\Omega)} \subseteq \Inv(\le^\opp)$. By construction $\overline{\Inv(\Omega)}$ is transitively closed. It remains to show that $\overline{\Inv(\Omega)}$ is transitively coclosed.

Suppose that $a < b < c$ and that $(a, c) \in \overline{\Inv(\Omega)}$. Then there exists a sequence $$a = x_1, x_2, \dots, x_n = c$$ such that for all $i$, we have $(x_i, x_{i+1}) \in \Inv(\preceq_i)$ for some $\preceq_i \in \Omega$.

Because $\le$ is a total order and $$a = x_1 \le \dots \le x_n = c,$$ there exists an $i$ such that $x_i \le b < x_{i+1}$.  If $x_i = b$ then $(b, x_{i+1}) \in \overline{\Inv(\Omega)}$ and $(b, c) \in \overline{\Inv(\Omega)}$ by transitivity. Otherwise, $x_i < b < x_{i+1}$ and $(x_i, x_{i+1}) \in \Inv(\preceq_i).$ As $\Inv(\preceq_i)$ is transitively coclosed, either $(x_i, b) \in \Inv(\preceq_i)$ or $(b,  x_{i+1}) \in \Inv(\preceq_i)$. Then $(a, b) \in \overline{\Inv(\Omega)}$ or $(b, c) \in \overline{\Inv(\Omega)}$ respectively. Therefore $\overline{\Inv(\Omega)}$ is transitively coclosed.

Hence there exists a total order $\preceq_\Omega$ such that $\overline{\Inv(\Omega)} = \Inv(\preceq_\Omega)$. Now suppose that $\preceq_\Omega'$ is any other common upper bound to all $\preceq \in \Omega$. Then necessarily $\Inv(\Omega) \subseteq \Inv(\preceq_\Omega')$, and as $\Inv(\preceq_{\Omega}')$ is transitively closed, $\Inv(\preceq_\Omega) \subseteq \Inv(\preceq_\Omega')$.

\end{proof}

\subsection{312-avoiding total orders and ornamentations}

\begin{definition}
    We say that a total order $\preceq$ on $X$ is \dfn{312-avoiding} (relative to $\le$) if there does not exist $a, b, c \in X$ such that $a < b < c$ but $c \prec a \prec b$. 
    
    Equivalently, $\preceq$ is 312-avoiding if whenever $(a, c) \in \Inv(\preceq) \tand (b, c) \in \Inv(\preceq)$ then $(a, b) \in \Inv(\preceq)$.

    We define  $\Weak^{312}(X_\le)$ to be the subposet of $\Weak(X_\le)$ of 312-avoiding total orders.

\end{definition}

\begin{definition}
    A subset $C \subseteq X$ is called \dfn{convex} if for all $a, b, c \in X$ with $a \le b \le c$ if $a, c \in C$ then $b \in C$.
    
    A \dfn{left segment} $R \subseteq X$ is a convex subset of $X$ with a minimal element.

\end{definition}

\begin{example}
We give some examples and non-examples of left segments.
\begin{itemize}
    \item Any closed interval $[a,b]$, half-open interval $[a,b)$, and ray $[a, \infty)$ is a left segment.
    \item The set $\{x \in \mathbb Q \mid 0 \le x < \sqrt{2}\}$ is a left segment but not a half-open interval. 
    \item The open interval $(0, 1) \subseteq \mathbb R$ is convex but not a left segment.
\end{itemize}

\end{example}

\begin{definition}
The left segment pointed building set $\BB(X_\le)$ is the set of left segments of $X$, pointed at their left endpoint. 
\end{definition}

\begin{proposition}
For all totally ordered sets $(X, \le)$ we have $\Weak^{312}(X_\le) \simeq \OO(\BB(X_\le))$.
\end{proposition}

\begin{proof}
We define a map $\Psi: \Weak^{312}(X_\le) \to \OO(\BB(X_\le))$ by $$\Psi(\preceq)(a) := \left(\{b \in X \mid (a, b) \in \Inv(\preceq)\} \cup \{a\}, a \right).$$
We verify that $\Psi$ is well-defined. First observe that if $a < b < c$ and $c \in \Psi(\preceq)(a)$ then either $(a, b) \in \Inv(\preceq)$ or $(b, c) \in \Inv(\preceq)$. But if $(b, c) \in \Inv(\preceq)$ then $(a, b) \in \Inv(\preceq)$ as $\preceq$ is 312-avoiding. Hence $(a, b) \in \Inv(\preceq)$ so $b \in \Psi(\preceq)(a)$. Hence $\Psi(\preceq)(a)$ is a left-segment.

Now suppose that $b \in \Psi(\preceq)(a)$. Then if $c \in \Psi(\preceq)(b)$ we have $(a, b) \in \Inv(\preceq)$ and $(b, c) \in \Inv(\preceq)$. Hence by transitivity, $(a, c) \in \Inv(\preceq)$ so $c \in \Psi(\preceq)(a)$. Therefore $\Psi$ is well-defined and is order-preserving by construction.

We can construct $\Psi^{-1}$ as follows: let $\rho \in  \OO(\BB(X_\le))$. We define a total order $\preceq_\rho$ whose inversion set is precisely $$\{(a, b) \in X \times X \mid a < b \tand b \in \rho(a)\}.$$ If $\Psi^{-1}$ is well-defined it is immediate that it is actually the inverse of $\Psi$ and is order-preserving. Observe that  $\Inv(\preceq_\rho)$ is transitively closed as $\rho$ is transitively closed and $\Inv(\preceq_\rho)$ is transitively coclosed and 312-avoiding as if $a < b < c$ and $(a, c) \in \Inv(\preceq_\rho)$ then $(a, b) \in \Inv(\preceq_\rho)$. 
\end{proof}

\subsection{The infinite, bi-infinite, and continuous Tamari lattices}
Recall from \cref{ex:TamariLatticeAndInfiniteTamariLattice} the $n$-th Tamari lattice $\Tam_n$, whose pointed building set is $\BB([n]_\le)$, the infinite Tamari lattice $\Tam_\infty$, whose pointed building set is $\BB(\Z^+_\le)$, and the bi-infinite Tamari lattice whose pointed building set is $\BB(\Z_\le)$. As a direct corollary of \cref{thm:inverse_limit_general}, we obtain:

\begin{corollary}
   The infinite Tamari lattice $\Tam_\infty \simeq \InverseLimit \Tam_n$.
\end{corollary}

Furthermore, let $$\BB'_n := \begin{cases}
    \BB(\{-k, \dots, k\}_\le) & n = 2k+1 \\ 
    \BB(\{-k+1, \dots, k\}_\le) & n = 2k
\end{cases}$$

Observe that $\OO(\BB'_n) \simeq \Tam_n$. Then under this (different) direct limit of Tamari lattices we obtain the following corollary.

\begin{corollary}
   The bi-infinite Tamari lattice $\Tam_{\pm \infty} \simeq \InverseLimit \OO(\BB'_n)$.
\end{corollary}

Equivalently, in order to obtain $\Tam_{\pm \infty}$, one can take the projection maps $\Tam_{n+1} \mapsto \Tam_n$ by removing alternating end points of the interval. 

\begin{remark}

In~\cite{dehornoy2012tamari}, a direct limit of Tamari lattices is considered, but this fails to be a complete lattice.
\end{remark}

\subsection{The continuous Tamari lattice}

Recall from~\cref{ex:ContinuousTamariLattice} that the \emph{continuous Tamari lattice} has $\IndexingSet$ equal to a closed interval in $\mathbb R$ and has for its pointed building set the collection of half-open intervals of the form $[a, b)$ where $a\le b$ pointed at their left endpoint. (And where $[a,a) := \{a\}$ by convention.)

\begin{example}
    Focusing on the example, $\IndexingSet = [0,1]$, let $\rho \in \OO(\BB)$ and let $[0,1]^{[0,1]}$ be the set of functions from $[0,1] \to [0,1]$. We can define an injective map $F: \OO(\BB) \to [0,1]^{[0,1]}$ by $F(\rho)(a) = b$ where $\rho(a) = [a, b)$. The image of this map is precisely all functions $f: [0,1] \to [0,1]$ such that
    \begin{itemize}
        \item for all $x \in [0,1]$ $x \le f(x)$
        \item for all $x, y \in [0,1]$ if $x < y$ and $y < f(x)$ then $f(y) \le f(x)$.
    \end{itemize}

    Examples of such functions include $f(x) = x, f(x) = 1, f(x) = |x-\frac12| + \frac 12$.

    Furthermore, if we order $f \le g$ if and only if $f(x) \le g(x)$ for all $x \in [0,1]$, then this is a poset embedding. We note that these functions $f$ are not necessarily continuous or even measurable as any function of the form $$f(x) = \begin{cases} x & x \in A \\ 1 & x \notin A \end{cases}$$ is in the image of $F$ for any $A \subseteq [0,1]$.
\end{example}

\begin{proposition}
\label{prop:continuousTamariNotIsomorphic}
    The continuous Tamari lattice $\OO(\BB)$ is not isomorphic to $\OO(\BB([0,1]_\le))$.
\end{proposition}

The intuition for why these two lattices are not isomorphic is that the pointed building set for the continuous Tamari lattice only includes half-open intervals whereas $\BB([0,1]_\le)$ also contains closed intervals. In applications, one may prefer the continuous Tamari lattice because it is more uniform and because of this identification with functions.

\begin{proof}[Proof of \protect{\cref{prop:continuousTamariNotIsomorphic}}]
Let $\rho \in \OO(\BB)$ be defined by $$\rho(i) = \begin{cases} ([0,1), 0) & i = 0 \\ (\{i\}, i) & \text{otherwise}. \end{cases}$$
The principal order ideal generated by $\rho$ is isomorphic to $[0,1]$ which does not contain any cover relations. However, any non-trivial principal order ideal in $\OO(\BB([0,1]_\le))$ contains a cover relation, namely one of the form $\rho_1 \precdot \rho_2$ where $$\rho_1(i) = \begin{cases} ([0,a), 0) & i = 0 \\ (\{i\}, i) & \text{otherwise} \end{cases} \qquad \qquad \rho_2(i) = \begin{cases} ([0,a], 0) & i = 0 \\ (\{i\}, i) & \text{otherwise} \end{cases}  $$
for some $a \in [0,1]$.

\end{proof}

\section{Ornamentation lattices under group actions}
\label{sec:group_actions}

Let $G$ be a group acting on $\IndexingSet$.
\begin{definition}
Let $\BB$ be a pointed building set. For $g \in G$ and $(S, i) \in \BB$, define $g \cdot (S, i) := (g\cdot S, g\cdot i)$ where $g \cdot S = \{g \cdot j \mid j \in S\}$. We say that $\BB$ is \dfn{$G$-invariant} if for all $g \in G$ and $(S, i) \in \BB$ we have $g\cdot (S, i) \in \BB$. An ornamentation $\rho \in \OO(\BB)$ is called {$G$-invariant} if for all $i \in \IndexingSet$ and $g \in G$ we have $\rho(g \cdot i) = g\cdot \rho(i)$. We denote the set of $G$-invariant ornamentations of $\BB$ by $\OO_G(\BB)$.
\end{definition}

\begin{example}
    Let $n \in \Z^+$ and let $G = \Z$ act on $\IndexingSet = \Z$ by translation by $n$, i.e., $a \cdot k= an+k$. Then for $\BB$ from the bi-infinite Tamari lattice (see \cref{ex:TamariLatticeAndInfiniteTamariLattice}), we have $\OO_G(\BB) \simeq \ATam_n$.
\end{example}

\begin{proposition} If $\BB$ is $G$-invariant, then $(\OO_G(\BB), \preceq)$ is a sublattice of $(\OO(\BB), \preceq)$.
\end{proposition}

\begin{proof}
    In~\cref{sec:Lattice_Properties}, we gave an explicit construction of the meet and join in $\OO(\BB)$. Both constructions are symmetric under the action of $G$.
\end{proof}

\begin{definition}
    If $\IndexingSet = \{-n, \dots -1, 1, \dots, n\}$, $G = \{1, -1\}$ acts by multiplication, and $\BB$ is $G$-invariant, we call $\BB$ \dfn{signed} and $\OO_G(\BB)$ the lattice of \dfn{signed ornamentations} of $\BB$. For clarity of notation, we will denote this by $\OO_{\pm}(\BB)$. 
\end{definition}
Compare this definition to the signed posets of Reiner~\cite{reiner1993signed}.

\begin{example}
\label{ex:signed_cycle}
    Let $\IndexingSet = \{-n, \dots, -1\} \cup \{1, \dots, n\}$ and define a digraph on $\IndexingSet$ with edge set $$E = \{(i, i+1) \mid i \in [n-1]\} \cup \{(-i, -i-1) \mid i \in [n-1]\} \cup \{(n, -1), (-n, 1)\}.$$ Then the digraphical pointed building set is signed. 

    We call $\OO_{\pm}(\BB)$ the \dfn{centrally symmetric affine Tamari lattice} and denote it $\CSymATam_n$.
\end{example}

\begin{example}
\label{ex:rotated_cycle}
    More generally, let $m, n$ be positive integers with $m \ge 2$ and let $D_{mn}$ be an oriented cycle with vertex set $\{1, \dots, mn\}$ and edge set $$E = \{(i, i+1) \mid 1 \le i \le mn-1\} \cup \{(mn, 1)\}.$$ Let $\BB_{m,n}$ be the digraphical pointed building set of $D_{mn}$ and let $G = {\Z/m\Z}$ act by rotation.  One immediately recovers $\CSymATam_n$ by setting $m = 2$.
\end{example}

For $\BB$ from \cref{ex:signed_cycle} and \cref{ex:rotated_cycle}, it will be convenient for  $(S, i) \in \BB$ to define $$|(S,i)| := \begin{cases} |S| & S \neq \IndexingSet \\ n+1 & S = \IndexingSet.\end{cases}$$

\begin{remark}
In fact, for integers $\ell, m \ge 2$ the map $\Psi: \OO_{\Z/\ell\Z}(\BB_{\ell,n}) \to \OO_{\Z/m\Z}(\BB_{m,n})$ defined by $$\Psi(\rho)(i) :=  \begin{cases}
     \rho(i) & |\rho(i)| \le n\\
     ([mn], i) & |\rho(i)| = n+1
 \end{cases}$$
 for $1 \le i \le n$ and extended by rotation is an isomorphism.
\end{remark}

\begin{figure}[ht]
    \centering
    \includegraphics[width=.9\linewidth]{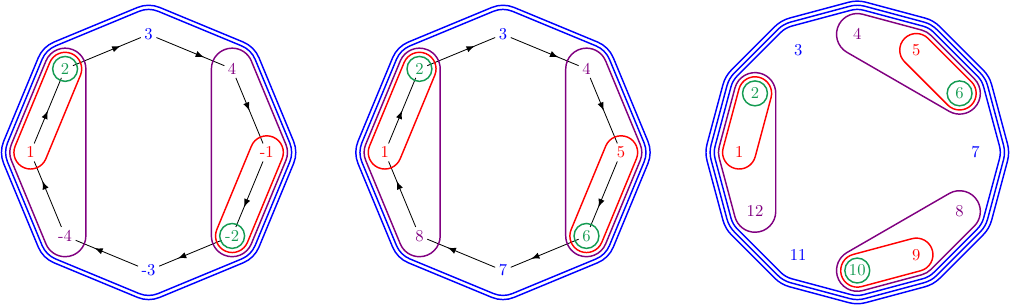}
    \caption{Three ornamentations invariant under group actions. Left: a signed ornamentation. Middle: the same ornamentation in $\OO_{\Z/2\Z}(\BB_{2,4})$. Right: the isomorphic copy of that ornamentation in $\OO_{\Z/2\Z}(\BB_{3,4})$. }
    \label{fig:GroupActionExample}
\end{figure}

The \emph{cyclic Tamari lattice} was defined in~\cite{barkley2025affine} in connection to combinatorial representation theory. For our purposes, we use the following definition of the cyclic Tamari lattice, which was shown in~\cite{barkley2025affine} to be equivalent to their original definition.
\begin{definition}[\cite{barkley2025affine}]
    An \dfn{arc} is a pair of integers $(i, j)$ satisfying
    \begin{itemize}
        \item $i < j$
        \item $1 \le i \le n$
        \item $j \le i+n$.
    \end{itemize}

    A \dfn{cyclic arc torsion class} is a collection of arcs $D$ satisfying
    \begin{itemize}
        \item if $(i,k) \in D$ and $i < j < k$, then $(i,j) \in D$;
        \item if $(i,j), (j,k) \in D$ and $k - i \le n$, then $(i,k) \in D$.
    \end{itemize}

    The \dfn{cyclic Tamari lattice} denoted $\CTam_n$ is the set of cyclic arc torsion classes ordered by inclusion.
\end{definition}

\begin{theorem}
\label{thm:CyclicTamariIsomorphism}
    For all integers $n \ge 2$, $\CSymATam_n \simeq \CTam_n$ .
\end{theorem}
\begin{proof}
    The map $\Psi: \CSymATam_n \to \CTam_n$ by $$\Psi(\rho) := \{(i, i+j-1) \mid 1 \le i \le n, 1 \le j < |\rho(i)|\}$$ is an isomorphism.
\end{proof}

\begin{remark}
In~\cite{barkley2025affine}*{Section 7}, $\ATam_n$ is described using triangulations of type $D_n$ and it is conjectured in~\cite{barkley2025affine}*{Section 10.2} that $\CTam_n$ can be described using centrally symmetric triangulations of type $D_{2n}$. By applying their map to our centrally symmetric ornamentations, we obtain centrally symmetric triangulations of type $D_{2n}$, thereby resolving their conjecture.
\end{remark}

In fact, our description of $\CTam_n$ allows us to give simple proofs for a few of the results of~\cite{barkley2025affine}. 

\begin{corollary}[\protect{\cite{barkley2025affine}*{Theorem 4.13}}]
    For all integers $n \ge 2$, $\CTam_n$ is semidistributive.
\end{corollary}
\begin{proof}
For all integers $n \ge 1$, $\CSymATam_n$ is a sublattice of $\ATam_{2n}$ which is semidistributive. 
\end{proof}

In a poset, the \dfn{length} of a chain is the number of elements in the chain. (Some authors define it as 1 less than the number of elements in the chain.)

\begin{theorem}[\protect{\cite{barkley2025affine}*{Theorem 8.4}}]
\label{thm:AffineTamariLength}
The maximal length of a chain in $\CSymATam_n$ is $\binom{n+1}{2}+1$.   
\end{theorem}
\begin{proof}
For $\rho \in \CSymATam_n$, define $f(\rho) = \sum\limits_{i = 1}^n|\rho(i)| - \binom{|\ \{i\ :\ \rho(i) = \IndexingSet \}\ |}{2}$.

Equivalently, let $X_\rho = \{i \in [n] \mid \rho(i) = \IndexingSet\}$ and $Y_\rho = [n] - X_\rho$ and define $X_\sigma \tand Y_\sigma$ similarly.
Then $$\begin{aligned} f(\rho) &= \sum\limits_{i \in X_\rho} |\rho(i)| - \binom{|X_\rho|}{2} + \sum\limits_{i \in Y_\rho} |\rho(i)|
\\&= |X_\rho|\left(n+1-\frac{|X_\rho|-1}{2}\right)  + \sum\limits_{i \in Y_\rho} |\rho(i)|.
\end{aligned}$$

We claim that if $\rho \prec \sigma$ then  $f(\rho) < f(\sigma)$. Suppose that $\rho \prec \sigma$ and define $$Z := \{i \in [n] \mid \rho(i) \subsetneq \sigma(i)\}.$$ 
$Z$ is non-empty as $\rho \prec \sigma$. If $i \in Y_\sigma$ then the contribution of $|\sigma(i)|$ to $f(\sigma)$ is strictly greater than the contribution of $|\rho(i)|$ to $f(\rho)$. Otherwise, suppose that $i \in Y_\rho$ and $i \in X_\sigma$. Then the difference of contributions of $|\sigma(i)|$ and $|\rho(i)|$ to $f(\sigma)$ and $f(\rho)$ is exactly $n-|X_\rho|+1- |\rho(i)|$. Observe that $|\rho(i)| \le n-|X_\rho|$ as one of the first $n-|X_\rho|$ vertices $j$ along the path starting at $i$ is in $X_\rho$, and therefore $j \notin \rho(i)$ (as $\rho(i) \neq (\IndexingSet, i)$). Hence in either case, the contribution to $f(\sigma)$ is strictly greater than the contribution to $f(\rho)$. As $Z$ is non-empty, $f(\rho) < f(\sigma)$.

Let $\hat 0, \hat 1 \in \CSymATam_n$ with $\hat 0$ minimal and $\hat 1$ maximal. To complete the proof, it suffices to find a chain of length $f(\hat 1) - f(\hat 0) +1 = n(n+1) - \binom{n}{2} - n+1 = \binom{n+1}{2} + 1$. For $1 \le i \le j \le n$, we define  $\rho_{i,j} \in \CSymATam_n$ by $$\rho_{i,j}(k) := \begin{cases}
    (\IndexingSet, k) & 1 \le k < i  \\
    (\{i, \dots, j\}, k) & k = i \\
    (\{k\}, k) & k > i
\end{cases}$$
for $1 \le k \le n$ and for $-n \le k \le -1$ $\rho_{i,j}(k) = -\rho_{i,j}(-k)$.

Observe that if $i < i'$ or $i = i' \tand j < j'$ then $\rho_{i,j} \prec \rho_{i', j'}$. Finally, $\rho_{n, n} \precdot \hat 1$, completing the proof. 
\end{proof}

\section{Geometric interpretations}
\label{sec:geometric_interpretations}

In this section, we introduce a geometric interpretation of ornamentations. 

\begin{definition}
    Let $\Phi$ be a set of vectors and let $\vec u, \vec v \in \Phi$. A subset $X \subseteq \Phi$ is called \dfn{closed in $\Phi$} if whenever $\vec u, \vec v \in X$ and $\vec w := \vec u + \vec v \in \Phi$ then $\vec w \in X$. We call $X$ \dfn{coclosed in $\Phi$} if whenever $\vec w \in X$ and $\vec w = \vec u + \vec v$ then $\vec u \in X$ or $\vec v \in X$ (or both).
\end{definition}

\begin{definition}
Let $\BB$ be a pointed building set on $\IndexingSet$. The set of vectors $\Phi_\BB \subseteq \R^{\IndexingSet}$ is defined by $$\Phi_\BB := \{e_j - e_i \mid \text{there exists } (S, i) \in \BB \st j \in (S-\{i\})\}$$ where $e_i$ is the standard basis vector corresponding to $i$.

For an ornamentation $\rho \in \OO(\BB)$ define $$V(\rho) := \{e_j - e_i \mid i \in \IndexingSet \tand j \in \rho(i) \st i \neq j\}.$$
\end{definition}

Observe that if $\IndexingSet = [n]$ then transitive closure of pointed building sets (axiom (2) in \cref{def:pointed_building_set}) implies that $\Phi_\BB$ is closed in the type-$A$ root system $\Phi_{A_{n-1}}$. Furthermore, transitive closure of ornamentations (axiom (2) in \cref{def:ornamentation}) is \emph{equivalent} to the fact that $V(\rho)$ is closed in $\Phi_\BB$ for all $\rho \in \OO(\BB)$. With this in mind, one can define an ornamentation $\rho$ to be biclosed if $V(\rho)$ is biclosed in $\Phi_\BB$.

Equivalently, an ornamentation $\rho \in \OO(\BB)$ is biclosed if for all $i \in \IndexingSet, k \in \rho(i)$, if there exist $(S, i),\, (T, j) \in \BB$ with $j \in S$ and $k \in T$ then either $j \in \rho(i)$ or $k \in \rho(j)$.

\begin{remark}
It seems interesting to study the subposet of biclosed ornamentations in $\OO(\BB)$.  We will denote this subposet $\Bicl(\BB)$.
    
\label{rem:biclosed}
    \begin{enumerate}
        \item All digraphical ornamentations of a directed tree are biclosed.
        \item Orient the edges of the complete graph $K_n$ from $i \to j$ with $i < j$ and let $\BB$ be the digraphical pointed building set of $K_n$. Then $\Phi_\BB$ is the set of positive roots of $\Phi_{A_{n-1}}$ and the poset of biclosed ornamentations is the weak order on $S_n$~\cites{dyer2019weak, barkley2023affine}.

        \item  Let $\BB$ be the graphical pointed building set of $K_n$. For $\rho \in \OO(\BB)$ define a binary operation $*_\rho$ by 
    $$i*_\rho j = 
        \begin{cases}
            i & j \notin \rho(i) \\
            j & j \in \rho(i)
        \end{cases}$$

        The operation $*_\rho$ is \dfn{quasitrivial}, i.e., for all $i, j \in [n]$ we have $i *_\rho j \in \{i, j\}$. A straightforward calculation shows that $*_\rho$ is associative if and only if $\rho$ is biclosed. In fact, one can show that $\rho \to *_\rho$ is a bijection from $\Bicl(\BB)$ to associative, quasitrivial operations on $[n]$, also known as quasitrivial semigroups~\cite{couceiro2019quasitrivial}. This construction recovers the bijection in~\cite{wang2025biclosed}. 
    \end{enumerate}

\end{remark}

In each of the examples in \cref{rem:biclosed}, $\Bicl(\BB)$ is a lattice. The first was proven in this paper, the second is well-known, and the third was proven in~\cite{dyer2014reflection} which is unpublished at the time of writing. It is not true in general that $\Bicl(\BB)$ is a lattice, see \cref{fig:NoJoinBiclosed}. However, experimentally it appears that $\Bicl(\BB)$ is a lattice for many natural choices of $\BB$. 

\begin{conjecture}
    For a finite (undirected) graph $G$, let $\BB$ be the graphical pointed building set of $G$. Then $\Bicl(\BB)$ is a lattice.
\end{conjecture}

\begin{figure}
    \centering
    \includegraphics[width=0.5\linewidth]{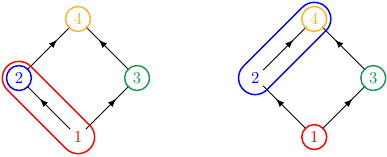}
    \caption{Two biclosed digraphical ornamentations that don't have a join in $\Bicl(\BB)$}
    \label{fig:NoJoinBiclosed}
\end{figure}

\section{Questions}
\label{sec:questions}

\begin{definition}
    A poset $P$ is \dfn{polytopal} if there exists a polytope $Q \subseteq \R^n$ and a vector $\vec u \in \R^n$ such that the 1-skeleton of $Q$ oriented in the direction $\vec u$ is the Hasse diagram of $P$.
\end{definition}

The weak order on permutations, the Boolean lattice, and the Tamari lattice are all well known polytopal lattices.
In~\cite{abram2025ornamentation}, it is shown that when $\TT$ is a rooted tree, the digraphical ornamentation lattice of $\TT$ is polytopal. In~\cite{barkley2025affine}, it is shown that $\ATam_n$ and $\CTam_n$ are always polytopal. 

\begin{question}
For graphical and digraphical $\BB$, when are $\OO(\BB), \OO_G(\BB), \Bicl(\BB)$ polytopal?
\end{question}

\begin{question}
    In this paper, we focused on graphical and digraphical pointed building sets. Are there other natural classes of pointed building sets with interesting combinatorics?
\end{question}

\begin{question}
One can view a pointed building set as a collection of lattices with compatibility conditions induced by their embeddings into the Boolean lattice. Can one develop the theory of ornamentation lattices without relying on this embedding? Can lattices other than the Boolean lattice be used?
\end{question}

\begin{remark}
One can remove the axiom that pointed building sets contain all singletons (axiom (1) of \cref{def:pointed_building_set}) and the axiom that ornamentations are pointed at their images (axiom (1) of \cref{def:ornamentation}) and still obtain join-semilattices. It may be interesting to study the combinatorics under these relaxed axioms. It may also be interesting to study a version where sets can have multiple distinguished elements.
\end{remark}

\begin{remark}
    Building sets have been highly influential in matroid theory. It would be interesting to study if pointed building sets have any application to matroids and in particular oriented matroids.
\end{remark}

\section{Acknowledgments}

Thanks to Vincent Pilaud, Felix Gelinas, Jose Bastidas, Colin Defant, Grant Barkley, Oliver Pechenik, and Sergey Fomin for insightful comments and conversations.

\bibliographystyle{amsrefs}
\bibliography{main}
\end{document}